\documentclass{article}
\usepackage{amsmath,amsfonts,latexsym}
\usepackage[english]{babel}
\usepackage[pdftex]{graphicx}
\usepackage{hyperref}
\usepackage{multicol}
\usepackage{amsmath}
\usepackage{amsthm}
\usepackage{amssymb}
\usepackage{amsfonts}
\usepackage{multicol}
\usepackage{color}
\usepackage{authblk}

\usepackage[all]{xy}

\usepackage{tikz}
    \usetikzlibrary{cd}
    \usetikzlibrary{shapes.misc}

\newtheorem*{theorem*}{Theorem}
\newtheorem{lemma}{Lemma}%[section]
\newtheorem{theorem}{Theorem}%[section]
\usepackage{hyperref}
\hypersetup{
colorlinks=true,       % false: boxed links; true: colored links
    linkcolor=blue,          % color of internal links
    citecolor=blue,        % color of links to bibliography
    filecolor=blue,      % color of file links
    urlcolor=blue           % color of external links
}

\begin{document}
 
\title{Extremal k-packings in compact non-orientable  surfaces}

\author{Ernesto Girondo\thanks{Partially supported by grant MTM2016-79497-P} and Cristian Reyes\thanks{This work was funded by the CONICYT PFCHA/Becas Chile 72180175}}

%\author[1]{Ernesto Girondo\thanks{ernesto.girondo@uam.es}}
%\author[2]{Cristian Reyes\thanks{cristianr.reyes@estudiante.uam.es}}
\date{Departamento de Matem\'aticas, Universidad Aut\'onoma de Madrid (Spain) \\
\ \\
ernesto.girondo@uam.es, cristianr.reyes@estudiante.uam.es}
%\affil[2]{Departamento de Matem\'aticas. Universidad Aut\'onoma de Madrid (Spain)}

\maketitle

\begin{abstract}
An extremal $k$-packing is a collection of $k$ mutually disjoint metric discs, embedded in a surface, whose radius is maximal for the given topology. We study compact non-orientable surfaces of genus $g\ge 3$ containing extremal $k$-packings.
\end{abstract}

\section{Introduction and statement of results}

 C. Bavard showed in \cite{Bavard_1996} that the radius $R$ of an embedded metric disc in a compact hyperbolic surface $X$ satisfies the inequality 
 $$\cosh R \le \frac{1}{2\sin \frac{\pi}{6-6\chi} }$$ 
 where $\chi$ is the Euler characteristic of $X$ and the metric under consideration is the natural hyperbolic metric. He also proved that the inequality becomes an equality precisely when $X$ is uniformized by a torsion-free Fuchsian group $K$ acting in the unit disc $\mathbb{D}$ such that at some point $z$  the Dirichlet domain
$$D(K,z):=\{w\in\mathbb{D}:\ d(w,z)\le d(w,\gamma(z)),\ \forall\gamma\in K\}$$ 
is a (hyperbolic) regular polygon with angle $2\pi/3$. The corresponding surfaces are known as  \emph{extremal surfaces}, and have been widely investigated in the literature both in the orientable and in the nonorientable case  (see \cite{Bavard_1996}, \cite{Girondo_Gonzalez-Diez_1999}, \cite{Bacher_Vdovina_2002},  \cite{Girondo_Nakamura_2007}, \cite{Nakamura_2016}) .

\

The first author studied the embedding of $k$ metric disks, the so called \emph{$k$-packings}, and he proved in \cite{Girondo_2018} that if  a set of $k\ge 1$ pairwise disjoint metric discs of fixed radius $R$ are embedded in a compact Riemann surface $X$ with centres at $p_1, \ldots, p_k$, then a fundamental domain for the Fuchsian group $K$ uniformizing $X$ can be obtained as the union $\bigcup\limits_{j=1}^kD(K,z_j,\{z_1,\ldots ,z_k\})$ of the Dirichlet cells $$D(K,z_j,\{z_1,\ldots ,z_k\}):=\{w\in\mathbb{D}:\ d(w,z_j)\le d(w,\gamma(z)),\ \forall\gamma\in K,\ \forall z\in\{z_1,\ldots ,z_k\}\}$$ where $z_j \in \mathbb{D}$ is a point in the fibre of the centre $p_j\in X$ by the canonical projection $\mathbb{D} \to \mathbb{D}/K\simeq X$, and $D(K,z_j,\{z_1,\ldots ,z_k\})$ contains the disc of radius $R$ and centre $z_j$. As a consequence, the  radius $R$ verifies $$\cosh R \le \frac{1}{2\sin \frac{k\pi}{12g+6k-12} }$$ and one can study \emph{extremal} $k$-\emph{packings}, namely, collections of $k$ disjoint metric discs whose radius gives an identity in the above inequality.

\

In this paper we extend the notion of $k$-extremal surface to the non-orientable case, where one has 

\begin{theorem*}If $R$ is the radius of a $k$-packing inside a compact non-orientable surface of genus $g\ge 3$, then $R$ must satisfy the inequality $$\cosh R \le \frac{1}{2\sin \frac{k\pi}{6g+6k-12}} .$$
\end{theorem*}

We refer to \cite{DeBlois_2018} for a  general result of this kind. Our approach to non-orientable $k$-extremal surfaces (namely, surfaces with a $k$-packing whose radius $R$ reaches the above bound) is via NEC groups. This way of addressing extremality allows us to find the characterization

\begin{theorem*}  A compact non-orientable  surface of genus $g\ge 3$ admits an extremal $k$-packing if and only if it is uniformized by a torsion free proper NEC group with index $12g+12k-24$ inside the extended triangle group $\Delta^{\pm}\left(2,3,\frac{6g+6k-12}{k}\right)$.
\end{theorem*}
 
In order to find conditions that guarantee the existence of non-orientable $k$-extremal surfaces of genus $g$ (see again \cite{DeBlois_2018} for another approach to this problem), we introduce what we call \emph{edge-grafting}, a new (at least to our knowledge) combinatorial machinery with topological flavor, and prove the following

\begin{theorem*} Given two integers $k\ge 1$ and  $g\ge 3$ there exist compact non-orientable  $k$-extremal surfaces of genus $g$ if and only if $k$ divides $6(g-2)$.
\end{theorem*}

We  next extend and improve  some of the results of \cite{Girondo_2018} to the non-orientable case. It is noteworthy that a surface can be simultaneously $k_1$- and $k_2$- extremal for two values $k_1\neq k_2$. We give a complete description of how this can happen in the following 

\begin{theorem*} If a compact non-orientable hyperbolic surface of genus $g$ is $k$-extremal for $m>1$ values $k_1, \ldots, k_m$ then $m=2$ and 
 one of the following holds:
\begin{itemize}
\item[1)] $\{k_1,k_2\}=\left\{2g-4,\displaystyle\frac{g-2}{2}\right\}$ 
\item[2)] $ \{k_1,k_2\}=\left\{6g-12,\displaystyle\frac{3g-6}{4}\right\}$ 
\end{itemize}
Conversely, for every even genus $g\ge 4$  there exists some compact non-orientable surface of genus $g$ which is $k_1$- and $k_2$- extremal for the values $k_1,k_2$ in case 1). Also, for every genus $g \equiv 2 \pmod 4$  there exists some compact non-orientable surface of genus $g$ which is $k_1$- and $k_2$- extremal for the values $k_1,k_2$ in case 2).
\end{theorem*}

We also address the possible existence of several extremal $k$-packings within a given surface, and we find

\begin{theorem*}
If $k\ge 1$ and $g\ge 3$ are such that $
\displaystyle\frac{6g+6k-12}{k} \in\{7,8,9,10,11,12,14,16,18,24,30\}
$
then there exists a $k$-extremal compact non-orientable surface with multiple extremal $k$-packings. 

If $\displaystyle\frac{6g+6k-12}{k} \not \in\{7,8,9,10,11,12,14,16,18,24,30\}$ then all compact non-orientable $k$-extremal surfaces have a unique extremal $k$-packing.
\end{theorem*}

We finish the paper showing the link between extremality in the orientable and the non-orientable case. The main result we find is 

\begin{theorem*} The canonical orientable double cover of a compact non-orientable $k$-extremal surface of genus $g$ is a compact  $2k$-extremal Riemann surface of genus $g-1$.
\end{theorem*}
 
\section{Disc packings in compact non-orientable hyperbolic surfaces} \label{sec:discpack}

We  devote this section to extend the notion of extremal $k$-packing to compact non-orientable surfaces. The following arguments mimic those of \cite{Girondo_2018}, and we include them for the sake of completeness.

\

Let $X$ be a compact hyperbolic non-orientable surface uniformized by a proper NEC group $K$, and let $z_1, \ldots,  z_k$ be representatives in the unit disk of a set of $k$ distinct points $p_1, \ldots, p_k$ in $X$. For every Dirichlet cell %$D_j=D(K,z_j,\{z_1,\ldots ,z_k\})$ 
$D_j=D(K,z_j,\{z_1,\ldots ,z_k\}):=\{w\in\mathbb{D}:\ d(w,z_j)\le d(w,\gamma(z)),\ \forall\gamma\in K,\ \forall z\in\{z_1,\ldots ,z_k\}\}$
whose union $\bigcup_{j=1}^k D_j$ constitute a fundamental domain for $K$, the proportion between the area of a circle of radius $r$ inside $D_j$ and the area $A_j$ of the cell must satisfy 
\begin{equation} \label{eq:Boroczky}
\frac{2\pi(\cosh r -1)}{A_j}\le\frac{3\alpha_r(\cosh r -1)}{\pi-3\alpha_r}
\end{equation}
 where $\alpha_r$ is the angle of an equilateral triangle with side $2r$. This was proved by K. Boroczky in \cite{Boroczky_1978}. 
 
 Therefore 
 $$A_j\ge \frac{2\pi(\pi-3\alpha_r)}{3\alpha_r}$$ 
 and, since the right hand side of this inequality does not depend on $k$, we have $$\sum\limits_{j=1}^kA_j\ge k\times \frac{2\pi(\pi-3\alpha_r)}{3\alpha_r}.$$ 
 
 The sum on the left is precisely the area of a fundamental domain for $K$, and this group uniformizes a compact non-orientable  surface of genus $g$, so 
 $$2\pi(g-2)\ge \frac{2k\pi(\pi-3\alpha_r)}{3\alpha_r}$$
  from which we obtain  $$\alpha_r\ge \frac{k\pi}{3g+3k-6}.$$ 
  
  A bit of hyperbolic trigonometry applied to the equilateral triangle with angle $\alpha_r$ and side $2r$ shows 
  $$\cosh r \sin \alpha_r =\cos \alpha_r/ 2$$ 
  and finally we get 
  $$\cosh r \le \frac{1}{2\sin \frac{k\pi}{6g+6k-12} }.$$

\begin{theorem} \label{th:bound} If $R$ is the radius of a $k$-packing inside a compact non-orientable  surface of genus $g\ge 3$, then $R$ must satisfy the inequality $$\cosh R \le \frac{1}{2\sin \frac{k\pi}{6g+6k-12} }.$$
\end{theorem}

A $k$-packing whose radius reaches this bound is called an \emph{extremal $k$-packing}, and a surface containing an extremal $k$-packing is a \emph{(non-orientable)  $k$-extremal surface of genus $g$}. We refer to \cite{DeBlois_2018} for a more general version of Theorem \ref{th:bound} obtained using a different  language.

\

Recall for further use that the extremal radius for $k$-packings occurs if and only if there is an equality in equation (\ref{eq:Boroczky}) for every $j$. According to Boroczky \cite{Boroczky_1978} this happens precisely when each of the Dirichlet cells $D_j$  is a regular polygon with angle $\frac{2\pi}{3}$. One can easily show (see  \cite{Girondo_2018}) that the number of sides of these polygons is the same for every $j$.

\

The area of a compact non-orientable  $k$-extremal surface of genus $g$, which is $2\pi(g-2)$, must be the area of the union of $k$ regular $N$-gons with angle $2\pi/3$. Therefore $$2\pi(g-2)=k\left((N-2)\pi-N\frac{2\pi}{3}\right)$$ so, we must have 
\begin{equation} \label{eq:equation_k-extremal}
kN=6g+6k-12
\end{equation} 
hence 
\begin{equation} \label{eq:equation_k-extremal-k} 
N=\frac{6g+6k-12}{k}=6+\frac{6(g-2)}{k}.
\end{equation}

Moreover, the natural decomposition of the Dirichlet cells $D_j$ into $2kN$ triangles with angles $\pi/2, \pi/ 3$ and $\pi\slash N$ yields the following characterization of compact non-orientable $k$-extremal surfaces in terms of triangle groups:
 
\begin{theorem} \label{th:char-triang} A compact non-orientable surface of genus $g\ge 3$ admits an extremal $k$-packing if and only if it is uniformized by a torsion free proper NEC group with index $12g+12k-24$ inside the extended triangle group $\Delta^{\pm}\left(2,3,\frac{6g+6k-12}{k}\right)$.
\end{theorem}

\subsection{Primitive extremal surfaces} \label{sec:primitive}

For a given integer $N \ge 7$, the identity (\ref{eq:equation_k-extremal-k}) holds for infinitely many pairs $(k,g)$ distributed along a line $L_N$ in the $k$$g$-plane. A $k$-extremal surface of genus $g$ is uniformized by a group with a fundamental domain that can be decomposed as a reunion of $k$ $N$-gons if and only if the pair $(k,g)$ belongs to the line $L_N$, that can be described by the equation  
$$
g=2+\frac{k}{6}(N-6)
$$
We find 
\begin{equation} \label{eq:ln}
\left.
\begin{array}{c}
L_N= \left\{ (k,g)= \left( j, 2+ j(N-6)/6 \right), \ j \ge 1 \right\}  \mbox{ if } N \equiv 0 \pmod 6  \\
 \ \\
L_N= \left\{(k,g)=  \left( 6j, 2+ j(N-6) \right), \ j \ge 1 \right\} \mbox{ if } N \equiv \pm 1 \pmod 6   \\
 \ \\
L_N= \left\{  (k,g)=\left( 3j, 2+ j(N-6)/2 \right), \ j \ge 1 \right\}  \mbox{ if } N \equiv \pm 2 \pmod 6  \\
 \ \\
L_N= \left\{  (k,g)=\left( 2j, 2+ j(N-6)/3 \right), \ j \ge 1 \right\}  \mbox{ if } N \equiv 3 \pmod 6  
\end{array}
\right\}
\end{equation}

The following terminology will be useful in the rest of the paper. We call a compact non-orientable $k$-extremal surface of genus $g$  \emph{primitive} if one of the following conditions happens:
\begin{itemize}
\item $k=1$.
\item $k=2$ and $g$ is even.
\item $k=3$ and $g-2 $ is coprime to 3.
\item $k=6$ and $g-2$ is coprime to 6.
\end{itemize}

Given any number $N\ge 7$ there exists a well defined pair $(k_N, g_N)\in L_N$ such that compact non-orientable $k_N$-extremal surfaces of genus $g_N$ are primitive. More precisely,   $(k_N, g_N)$ is determined by setting $j=1$ in (\ref{eq:ln}), so that the primitive case corresponds to the lowest possible $k$ (or, equivalently, lowest possible $g$) for a given $N$. 

\

If $(\tilde{g}-2)/\tilde{k}=(g_N-2)/k_N$ for other values $\tilde{k} \neq k_N$, $\tilde{g}\neq g_N$, the pair $(\tilde{k},\tilde{g})$ belongs to the same line $L_N$ as $(k_N, g_N)$  does. The NEC groups uniformizing a $k_N$-extremal surface of genus $g_N$ or a $\tilde{k}$-extremal surface of genus $\tilde{g}$ are both contained in $\Delta^{\pm}(2,3,N)$, but the index $2\tilde{k}N$ in the imprimitive case is always a multiple of the index $2k_NN$ of the primitive case, as $k_N$ obviously divides $\tilde{k}$ (see (\ref{eq:ln})). 

\

The fundamental regions can be constructed as a connected reunion of regular $N$-gons of angle $2\pi/3$ in either case, and both groups are  generated by certain side-pairing transformations between edges of the $N$-gons lying in the boundary of the fundamental region they define (some other edges may lie in the interior). The side-pairing transformations must act in such a way that every point of the surface represented by a vertex of some of the $N$-gons must be the center of a complete $2\pi$ angular sector around it.

\section{Construction of primitive extremal surfaces via edge-grafting}

 In this section we construct an example of a primitive compact non-orientable $k_N$-extremal surface of genus $g_N$ for every $N \ge 7$ (for a completely different kind of construction, see \cite{DeBlois_2018}). Our use of NEC groups and fundamental domains provides a surprisingly simple way of ensuring the existence of extremal surfaces, with just four surfaces somehow \emph{generating} all the rest via certain combinatorial operations we call \emph{edge-grafting}.

\

Consider for example  the case $N=7$, and note that $(k_7, g_7)=(6,3)$. The picture labeled $X_7$ in Figure \ref{ejemplos} shows a specific example  of  a compact non-orientable  $6$-extremal surface of genus $3$. In Figure \ref{ejemplos}, and in the rest of similar figures in what follows, edges labeled with the same numbers are identified by an orientation preserving side-pairing transformation (opposite signs indicate that the side pairing is orientation reversing). The group generated by the side-pairings defining $X_7$ is, as explained in the previous section, contained in the triangle group $\Delta^{\pm}(2,3,7)$ with  index $2\times 6 \times 7=84$. 

\

The other  surfaces in the same Figure \ref{ejemplos}, namely $X_8, X_9$ and $X_{12}$,  are examples of compact non-orientable  $k$-extremal surfaces of genus $g$ for the pairs $(k,g) =(3,3)$, $(2, 3)$ and $(1, 3)$ respectively. These are the primitive pairs $(k_N,g_N)$ for $N=8, 9$ and $12$. One can look for these examples by directly  constructing the appropriate side-pairings (in such a way that every class of vertices is surrounded by a full $2\pi$-angular sector) or, more generally, one can try to find a subgroup of the corresponding triangle group with the required properties (the LowIndexSubgroup routine of GAP is very helpful for this).

\

\begin{figure}[!h]
\begin{center}
\begin{tabular}{ccc}
\includegraphics[width=0.35\textwidth]{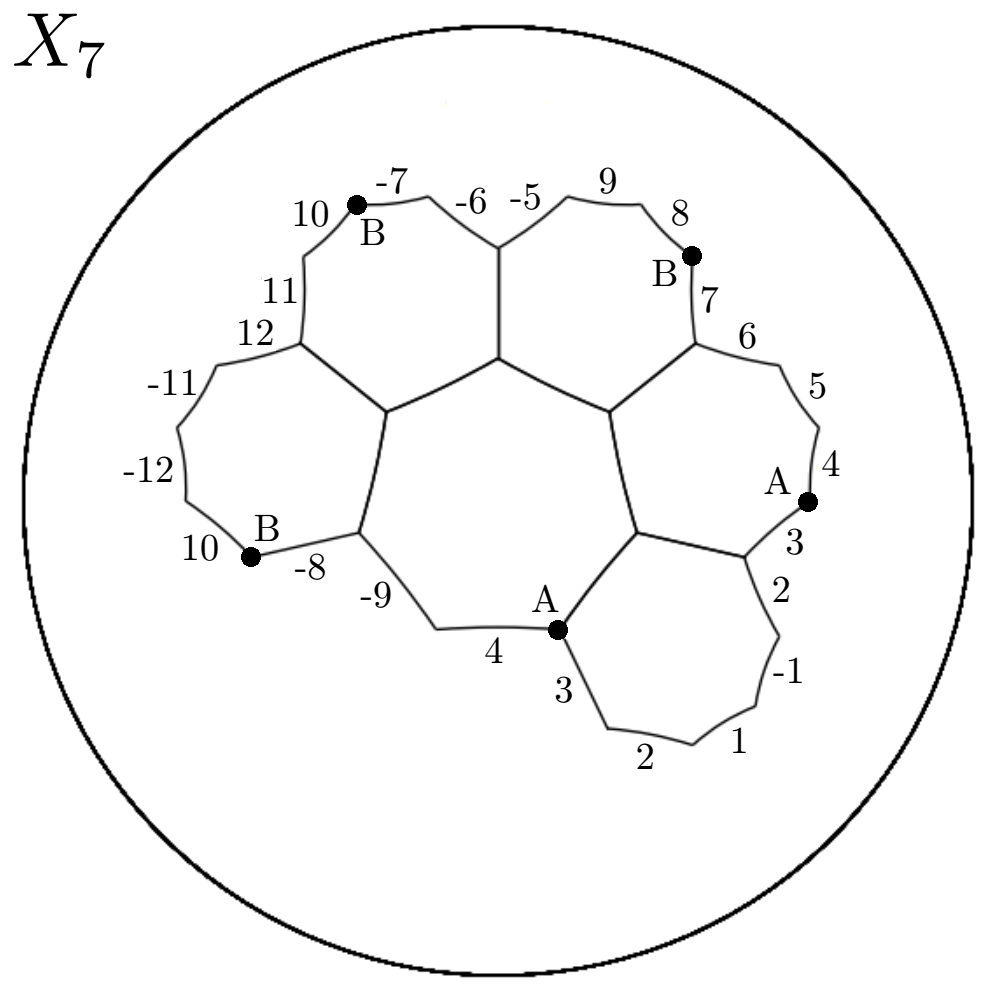} & \quad &  \includegraphics[width=0.35\textwidth]{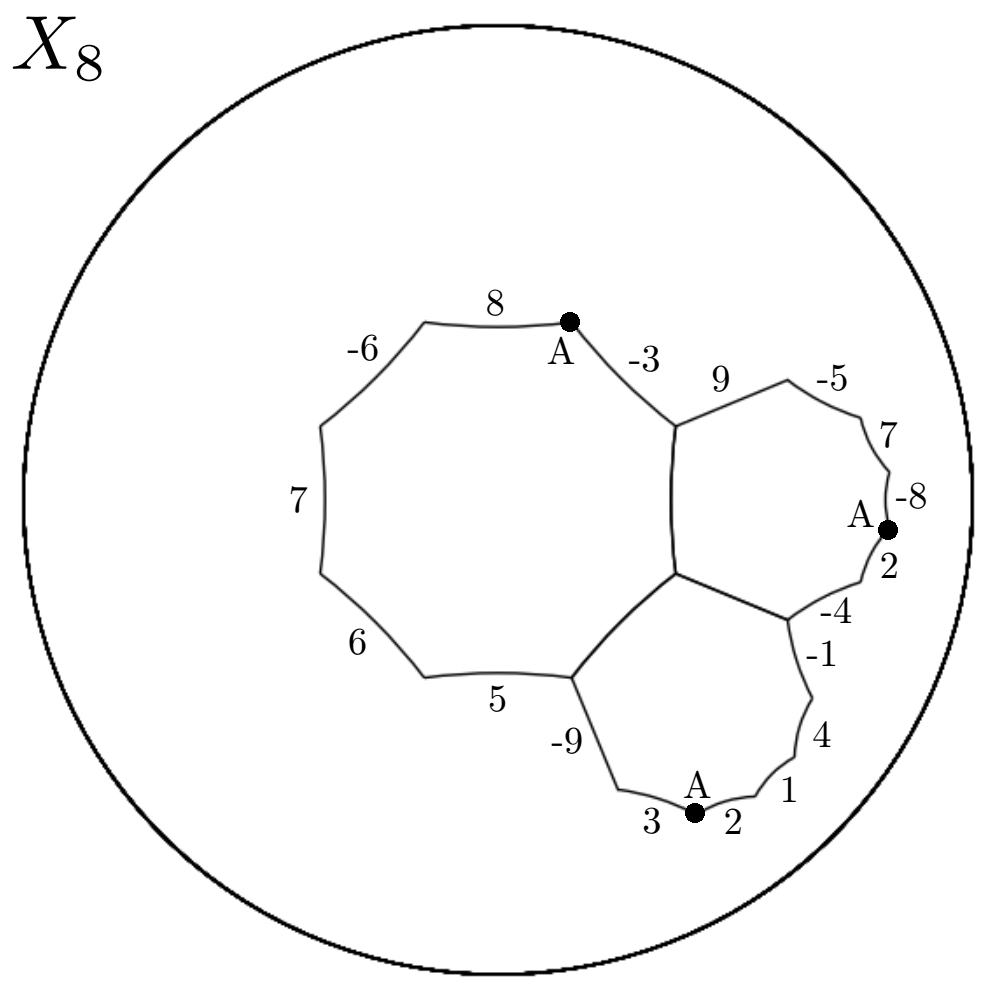} \\
\ & & \\
\includegraphics[width=0.35\textwidth]{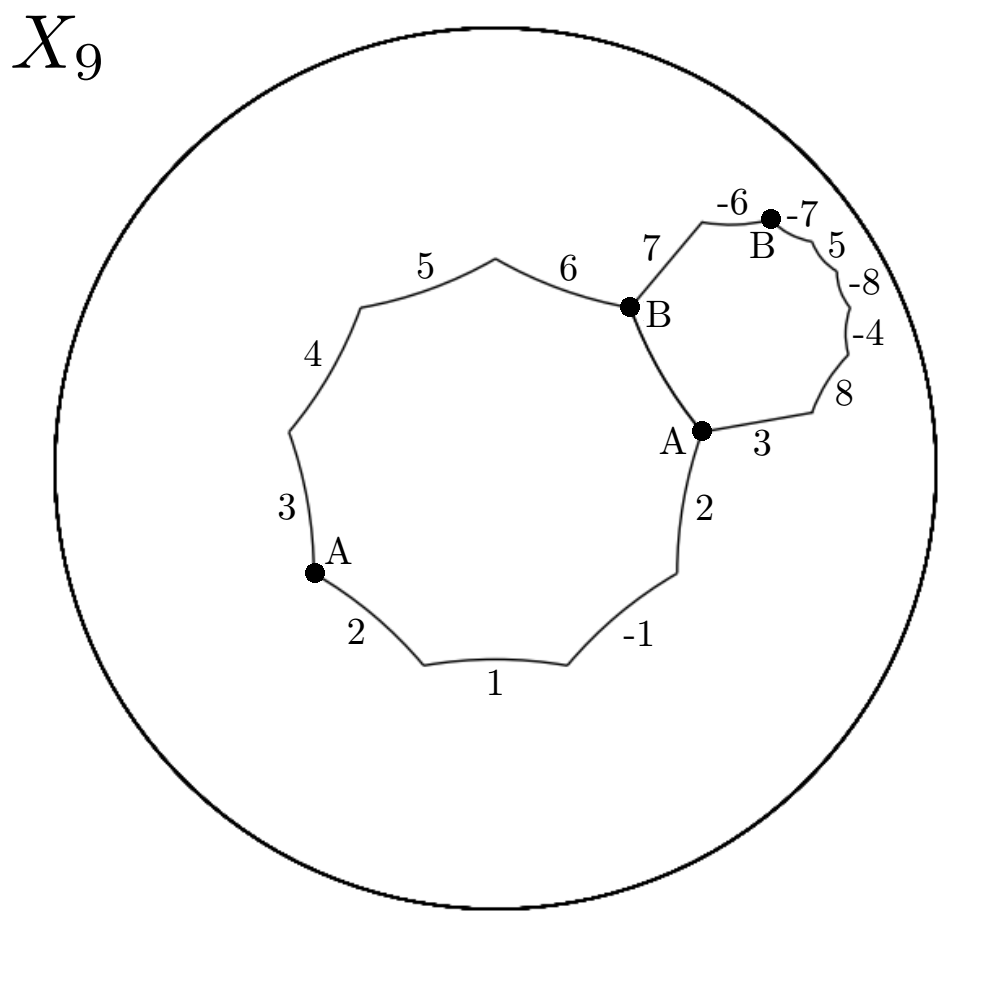} & \quad &\includegraphics[width=0.35\textwidth]{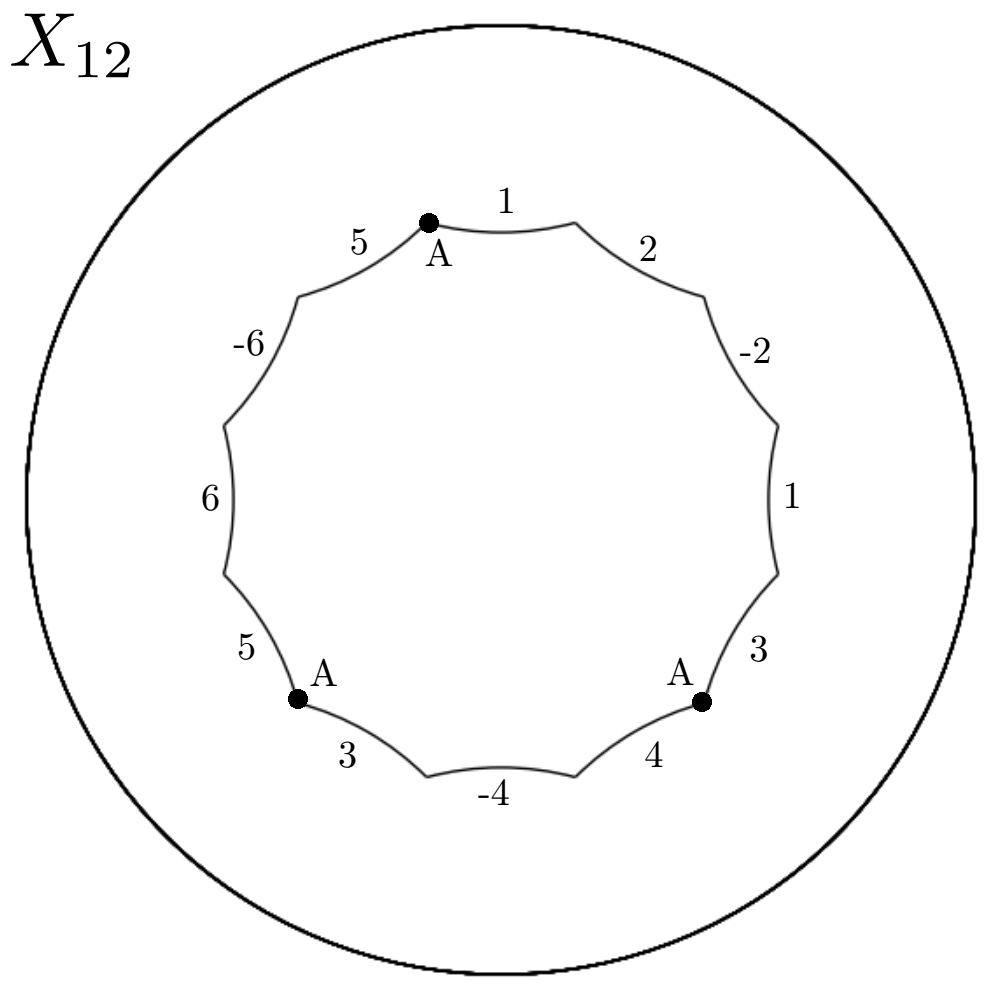} 
\end{tabular}
\caption{The four fundamental primitive extremal surfaces. From top to bottom, left to right, they are $k$-extremal and have genus $g$, for the 
values $(k,g)=(6,3), (3,3), (2,3)$ and $(1,3)$.}\label{ejemplos}
\end{center}
\end{figure}

Look now at the surface $X_8$ Figure \ref{ejemplos}, and consider the vertex-cycle marked as A. 
We can perform what we call an \emph{edge-grafting at A}, which topologically means  inserting  three pairs of new edges that will get paired as indicated at the top of Figure \ref{graft} , while we keep the old side-pairings between the old edges. This way we can replace the three original 8-gons by three 10-gons, still regular with angle $2\pi / 3$. In other words, the edge grafting procedure EG1 described at the top of Figure \ref{graft}, applied to the vertex $P=A$ in $X_8$, produces the surface $X_{10}$ in Figure \ref{ej2310}, which is nothing else than a $3$-extremal surface of genus 4, uniformized by an index 60 subgroup of the triangle group $\Delta(2,3,10)$. Note that $(k_{10},g_{10})=(3,4)$, so that $X_{10}$ is in fact a primitive extremal surface.

\

\begin{figure}[!h]
\begin{center}
\begin{tabular}{c}
\includegraphics[width=0.6\textwidth]{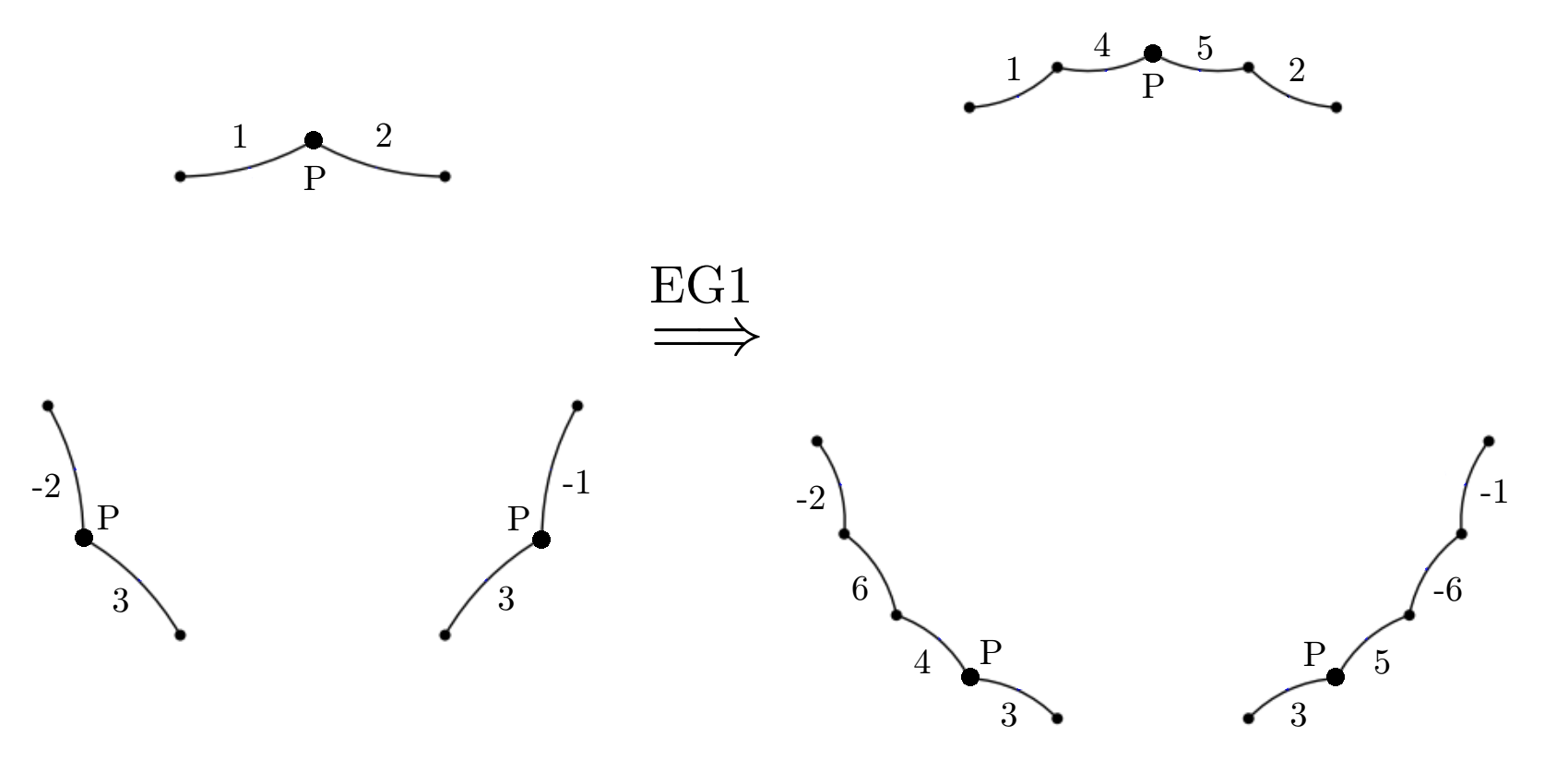} \\
\  \\
\includegraphics[width=0.6\textwidth]{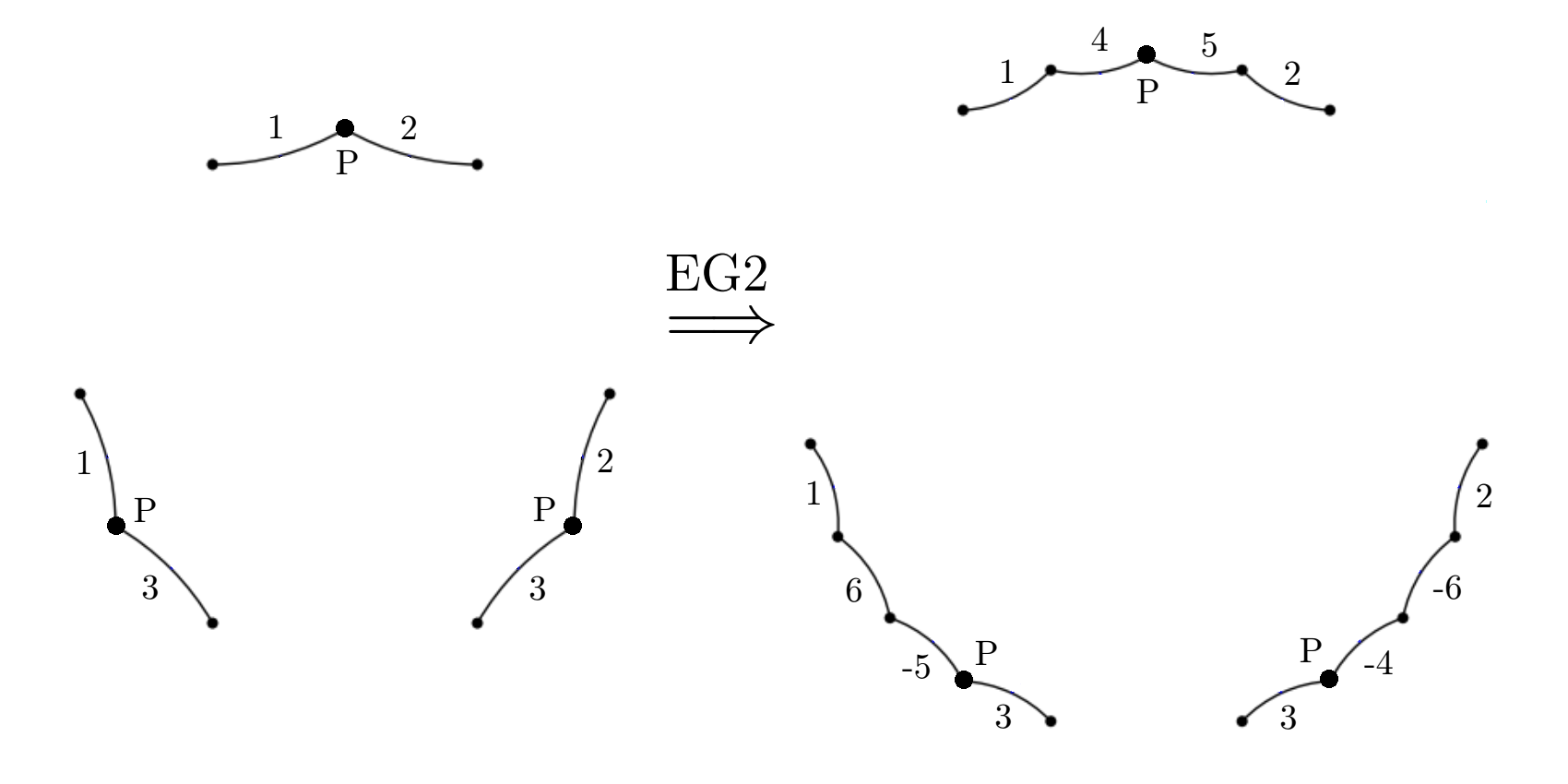} \\
\ \\
\includegraphics[width=0.6\textwidth]{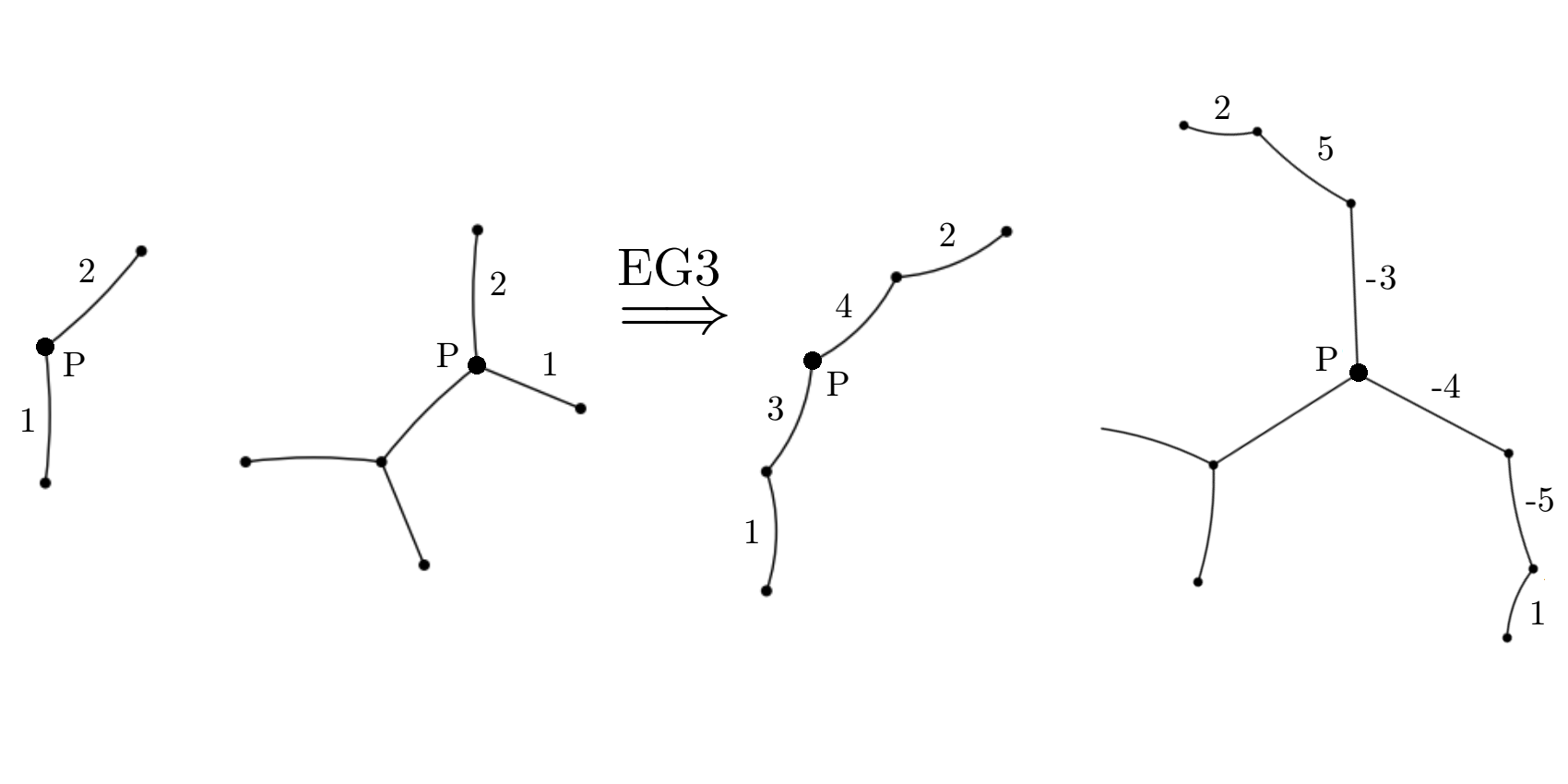} \\
\  \\
\includegraphics[width=0.6\textwidth]{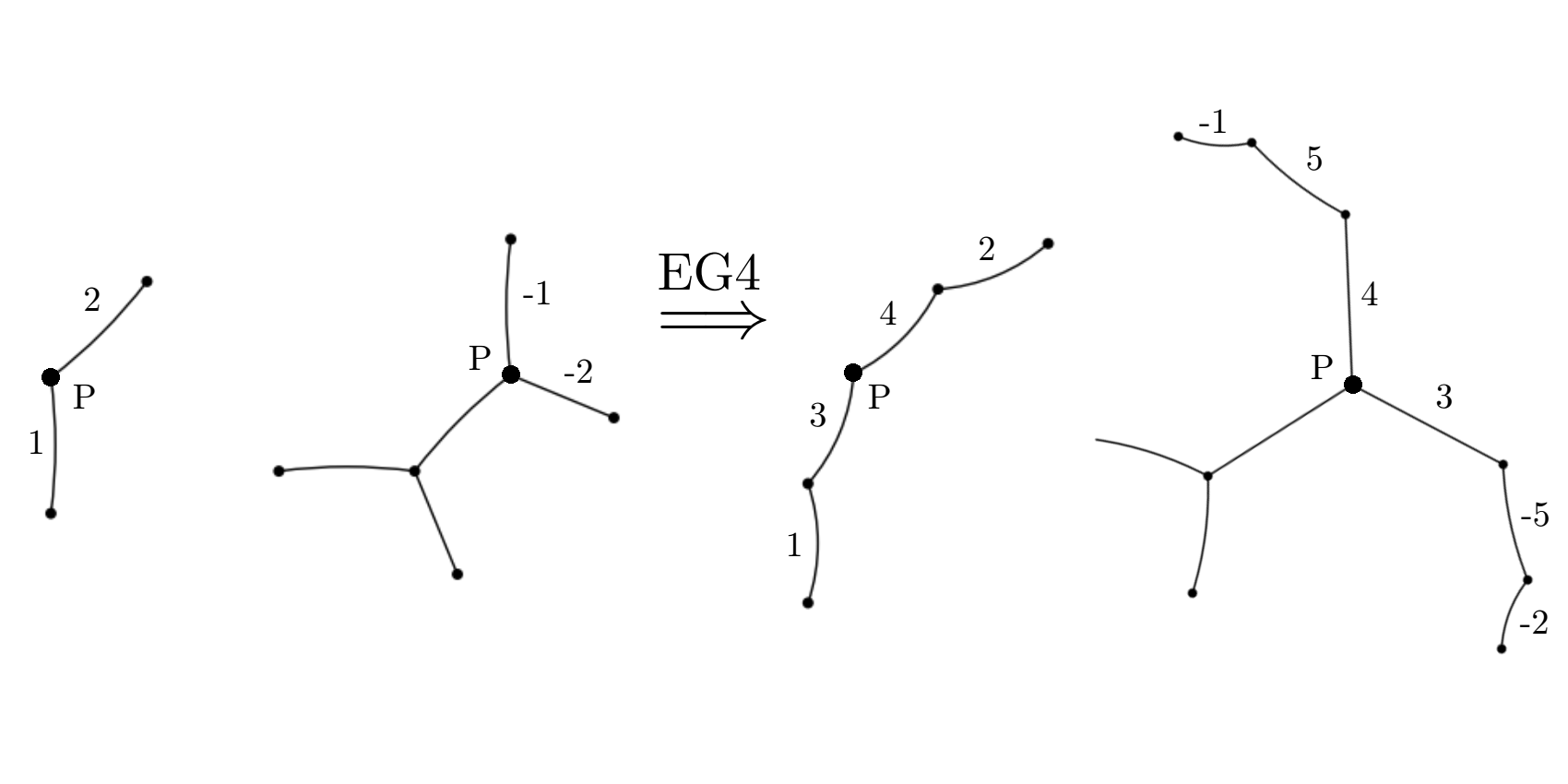} 
\end{tabular}
\caption{Four different edge grafting procedures at a vertex-cycle $P$.}\label{graft}
\end{center}
\end{figure}

\begin{figure}[!h]
\begin{center}
\includegraphics[width=0.35\textwidth]{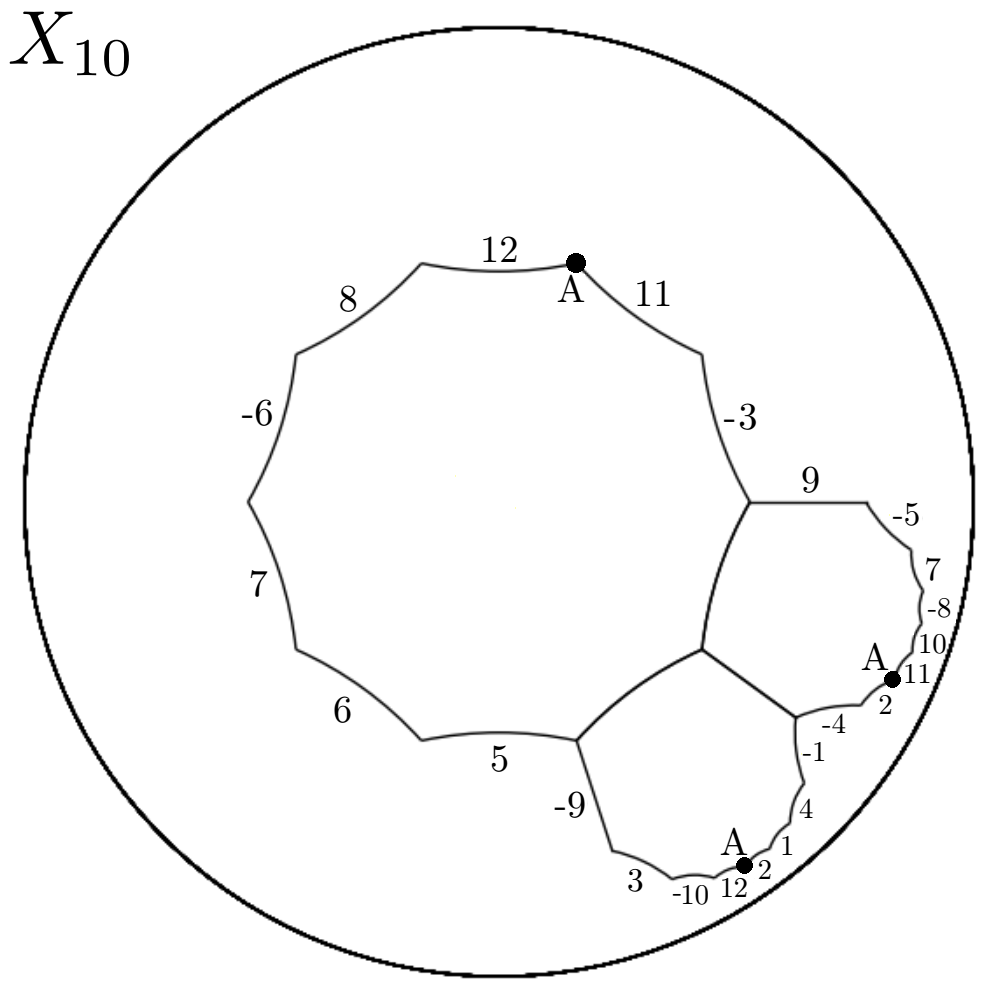} 
\caption{The surface $X_{10}$ obtained via edge-grafting EG1 at the vertex-cycle $A$ of $X_8$.}\label{ej2310}
\end{center}
\end{figure}

The second row of Figure \ref{graft} indicates schematically a second edge grafting EG2 that can be done at 
the vertex-cycle A of $X_{10}$. The result is a $3$-extremal surface of genus 5, uniformized by an index 72 subgroup of the triangle group $\Delta(2,3,12)$. Note that  $(k_{12},g_{12})=(1,3)$, so that this surface is imprimitive (on the other hand, we already have in Figure \ref{ejemplos}  a primitive surface $X_{12}$ for the case $N=12$).

\

There are two other edge grafting procedures that are essential in what follows: we call them EG3 and EG4, and we describe in Figure \ref{graft} how  they are defined. In contrast to EG1 and EG2, EG3 and EG4 involve a vertex $\mathrm{P}$ that can be seen only twice in the fundamental region we are dealing with: once at the boundary of a single polygon (angle $2\pi /3$) and once at a common edge of two polygons (angle $4\pi/3$). Note that, as it happens with EG1 and EG2, the procedures EG3 and EG4 can be performed consecutively in any order, since the resulting configuration of one of them is the starting one of the other. Both EG3 and EG4 could be done, respectively, at the vertices A and B of the surface $X_9$ shown in Figure \ref{ejemplos} (a primitive 2-extremal surface of genus 3). The result, $X_{15}$, is a primitive 2-extremal surface of genus 5.

\

Another convenient combination of edge graftings is useful in order to produce a primitive surface for the case $N=11$ from the primitive surface $X_7$ of  Figure \ref{ejemplos}. The adequate combination consists of doing EG3 at the vertex A and EG1 at the vertex B, which produces an increase of two edges in every polygon, thus giving a 6-extremal surface with $N=9$ (genus 5), which is imprimitive because $(k_9, g_9)= (2,3)$. But we could now combine EG4 and EG2 in this surface to obtain $X_{11}$, a primitive 6-extremal surface of genus 7 ($N=11$).

\

Summarizing, we  have already constructed a primitive extremal surface $X_N$ for the six basic cases $7\le N \le 12$. Further edge-grafting allows us to extend to arbitrary $N$. The point now is to recall again that EG1 and EG2 can be performed consecutively, and the same is true for EG3 and EG4. The primitive examples for any $N\equiv \pm 1 \pmod 6$ arise iterating what we did for creating $X_{11}$ from $X_7$. In the same way, we get primitive examples for $N\equiv \pm 2 \pmod 6$ by iterating  what we did for creating $X_{10}$ from $X_8$, and primitive examples for $N\equiv 3 \pmod 6$ can be obtained by iterating what we did with $X_9$ to obtain $X_{15}$. Finally, a sequence of alternate EG2 and EG1 edge graftings applied to $X_{12}$ produces the required primitive surface $X_N$ for any $N\equiv 0 \pmod 6$.

%%%%%%

\section{Existence of non-orientable $k$-extremal surfaces of genus $g$}\label{S3}

We are now ready to proof the following existence result (see \cite{DeBlois_2018} for a different approach to this):

\begin{theorem} \label{th:existence} Given two integers $k\ge 1$ and  $g\ge 3$ there exist compact non-orientable  $k$-extremal surfaces of genus $g$ if and only if $k$ divides $6(g-2)$.
\end{theorem}

\begin{proof}
The divisibility condition is necessary due to equation (\ref{eq:equation_k-extremal-k}). We already know from the previous section that the condition is sufficient for 
the primitive cases $(k,g)=(k_N, g_N)$, $N \ge 7$. All we have to do is to find the way for passing from primitive to imprimitive cases for each $N$, but this is possible due to Lemma \ref{le:sub} below. Recall that the index of the subgroup of $\Delta^{\pm}(2,3,N)$ uniformizing an imprimitive extremal surface is always a multiple of the index of the subgroups uniformizing primitive extremal surfaces for the same $N$.
\end{proof}

\begin{lemma}\label{le:sub} Suppose $K$ is a NEC group which uniformizes a compact non-orientable  surface of genus $g\ge 3$. Then for every $n\in\mathbb{N}$ there exists a subgroup $K_n\le K$ such that $[K:K_n]=n$ and $\mathbb{D}/K_n$ is also a compact non-orientable surface.
\end{lemma}

\begin{proof}
The signature of $K$ is $(g;-;[-],\{-\})$, so this group admits a presentation $$K=\left\langle d_1,d_2,\ldots ,d_{g}:\ \ \prod\limits_{i=1}^{g}d_i^2=1\right\rangle$$ where the $d_i$'s are glide reflections of infinite order. If $h:K\longrightarrow\mathbb{Z}$ is the surjective homomorphism determined by 
$$d_1\mapsto 1,\ d_2\mapsto -1,\ d_j\mapsto 0 \ \forall j\in\{3\ldots g\}$$ 
then the required $K_n$ can be taken as $h^{-1}(n\mathbb{Z})$.
\end{proof}

%%%%%%%%%%%%%%%%%%

 We note that the topology of a compact non-orientable  surface is completely characterized by the genus $g$, so Theorem \ref{th:existence} tells us when the existence of an extremal $k$-packing is possible on each of the different topologies. For instance, extremal $4$-packings exist if $g=4$ but they don't exist if $g=3$. On the other hand, it is certainly remarkable that certain $k$-packings do not encounter any topological obstruction:

\begin{theorem}\label{th3} There exists a compact non orientable $k$-extremal surface of genus $g$ for every $g\ge 3$
if and only if  $k\in\{1,2,3,6$\}.
\end{theorem}

\begin{proof} By Equation (\ref{eq:equation_k-extremal}), as $k$ must divide $6(g-2)$ for every $g\ge 3$, then $k$ must divide $6$, so $k\in\{1,2,3,6\}$.
\end{proof}

In the cases when topology puts obstructions to the existence of  extremal $k$-packings, such packings  exist still in infinitely many  topological classes of compact non-orientable surfaces. More precisely:

\begin{theorem} Given $k\ge 1$, there are compact non-orientable $k$-extremal surfaces for every genus $g\ge 3$ such that $g\equiv 2\left(mod \left(\frac{k}{gcd(k,6)}\right)\right)$.
\end{theorem}

\begin{proof} As $k$ must divide $6(g-2)$, then $\frac{k}{gcd(k,6)}$ must divide the product $\frac{6}{gcd(k,6)}\cdot  (g-2)$, but the numbers $\frac{k}{gcd(k,6)}$ and $\frac{6}{gcd(k,6)}$ are coprime, so $\frac{k}{gcd(k,6)}$ must divide $g-2$.
\end{proof}

On the contrary, we notice that fixing the topology of a non-orientable surface restricts the  possible extremal $k$-packings to a finite number:

\begin{theorem} Given $g\ge 3$, there is a finite number of values of $k$ such that there exist compact non-orientable $k$-extremal surfaces of genus $g$.
\end{theorem}

\begin{proof} By Equation (\ref{eq:equation_k-extremal}) we know that $k$ must divide $6(g-2)$ for such fixed $g$, so there are $\varphi(6(g-2))<\infty$ posibilities for $k$, where $\varphi$ is Euler's function.
\end{proof}

\section{Compact non-orientable extremal surfaces which admit two types of extremal packings}

 Suppose that a compact non-orientable surface $X=\mathbb{D}\slash K$ is $k_1$-extremal and $k_2$-extremal at the same time, with $k_1< k_2$. Then we have 
  \begin{equation} \label{eq:inc}
K\le \Delta_1^{\pm}:=\Delta^{\pm}\left(2,3,\frac{6g+6k_1-12}{k_1}\right),\ \ \ \ \ \ \ \ K\le \Delta_2^{\pm}:=\Delta^{\pm}\left(2,3,\frac{6g+6k_2-12}{k_2}\right)
\end{equation}
and the inclusions  have indexes $12g+12k_i-24,\ i\in\{1,2\}$. In these conditions $\Delta_1^{\pm}$ and $\Delta_2^{\pm}$ are commensurable, i.e. $\Delta_1^{\pm}\cap\Delta_2^{\pm}$ is contained in both groups $\Delta_1^{\pm}$ and $\Delta_2^{\pm}$ with finite index.

\

The canonical Fuchsian subgroup of $\Delta_1^{\pm}\cap \Delta_2^{\pm}$ is precisely its index $2$ subgroup $\Delta_1^{+}\cap \Delta_2^{+}$, therefore the inclusion $\Delta_1^{\pm}\cap \Delta_2^{\pm}\le \Delta_j^{\pm}$ will have finite index if and only if the inclusion $\Delta_1^{+}\cap \Delta_2^{+}\le \Delta_j^{+}$ has finite index ($j=1$ or 2), which means that $\Delta_1^{\pm}$ and $\Delta_2^{\pm}$ are commensurable if and only if $\Delta_1^{+}$ and $\Delta_2^{+}$ are. 

\

In particular, (cf. \cite{Girondo_2018}) we deduce that either
$$\Delta_1^{\pm}=\Delta^{\pm}(2,3,14), \ \Delta_2^{\pm}=\Delta^{\pm}(2,3,7) \quad  
\mbox{ or } \quad  \Delta_1^{\pm}=\Delta^{\pm}(2,3,18), \ \Delta_2^{\pm}=\Delta^{\pm}(2,3,9)$$
and therefore 
  $$K\le \Delta^{\pm}(2,3,18)\cap \Delta^{\pm}(2,3,9)=\Delta^{\pm}(3,3,9) \quad \mbox{ or } \quad K\le \Delta^{\pm}(2,3,14)\cap \Delta^{\pm}(2,3,7)=\Delta^{\pm}(3,3,7).$$
  
  A straightforward computation shows that, as a function of $g$, the correct values of $k_j$ can be expressed as $k_1=(g-2)/2$, $k_2=2g-4$ in the first case (with $g$ necessarily even), and $k_1=(3g-6)/4$, $k_2=6g-12$ in the second case (with $g$ necessarily congruent to 2 modulo 4).
  
 Since 
 $$ (4g-8)\cdot 9 = 2k_2\cdot 9=[\Delta^{\pm}(2,3,9):K]=[\Delta^{\pm}(2,3,9): \Delta^{\pm}(3,3,9)][\Delta^{\pm}(3,3,9):K]=4[\Delta^{\pm}(3,3,9):K],$$ 
 a genus $g$  compact non-orientable surface simultaneously $(g-2)/2$- and $(2g-4)$- extremal would be uniformized by a torsion free proper NEC group
 of index $(g-2)\cdot 9$ inside $\Delta^{\pm}(3,3,9)$. The lowest possible value of $g$, namely $g=4$, corresponds to the Hurwitz index of $\Delta^{\pm}(3,3,9)$ (the minimal index of a surface subgroup), which equals 18. In general, taking genus $g=2r$ with $r\ge 1$ would give index $(2r-2)\cdot 9=18(r-1)$, a multiple of the Hurwitz index.

\
 
According to \cite{Izquierdo_1993} there always exist torsion free proper NEC subgroups of the Hurwitz index, but the Kleinian surface they uniformize may have or not have boundary. If we  define $K_{18}$ as the NEC group generated by the side-pairing transformations given in the hyperbolic $18$-gon in Figure \ref{ejemplos33n} (left), the quotient surface  $\mathbb{D}/K_{18}$ has empty boundary and defines indeed a genus $g=4$ surface which is $k_1$- and $k_2$- extremal for  the case $k_1=(g-2)/2=1$, $k_2=2g-4=4$. Higher index subgroups of $\Delta^{\pm}(3,3,9)$ can now be constructed using Lemma \ref{le:sub}.
  
  \
  
\begin{figure}[!htbp]
\begin{center}
\begin{tabular}{cc}
\includegraphics[width=0.45\textwidth]{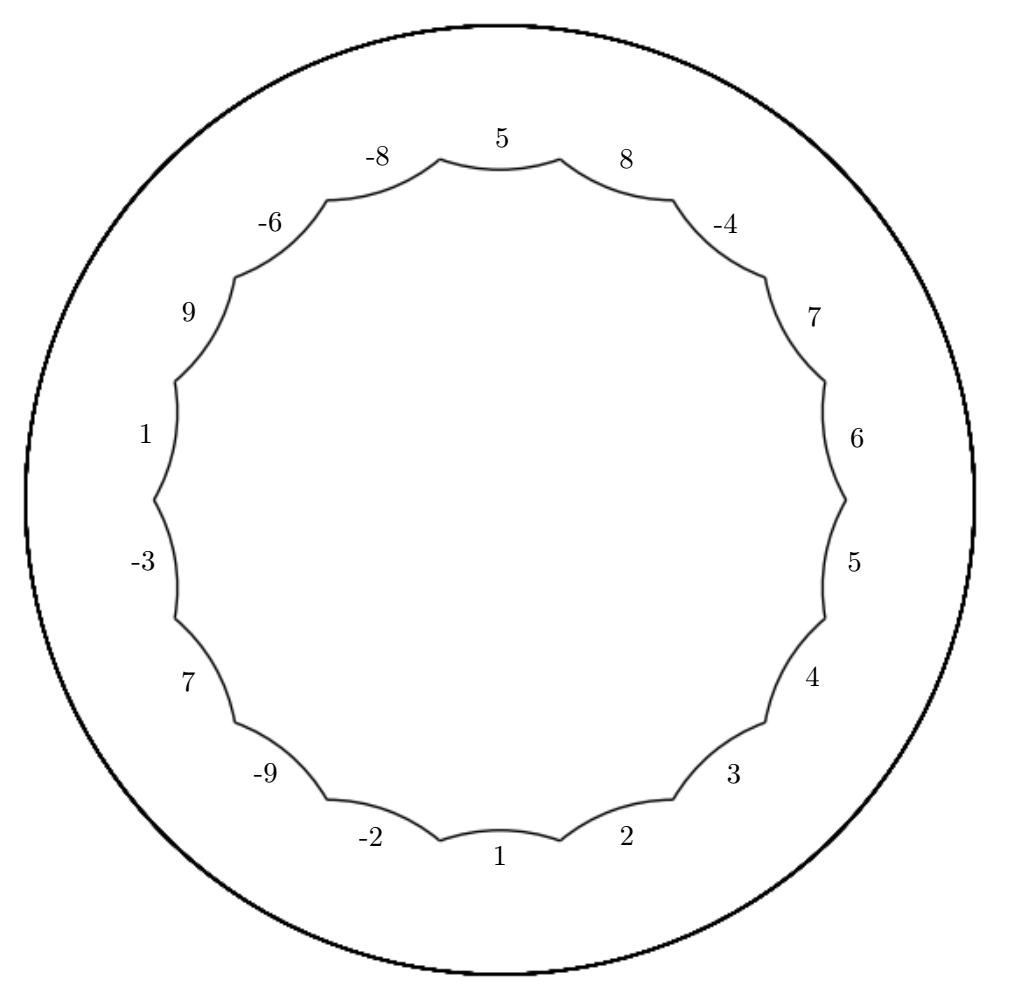} & \includegraphics[width=0.45\textwidth]{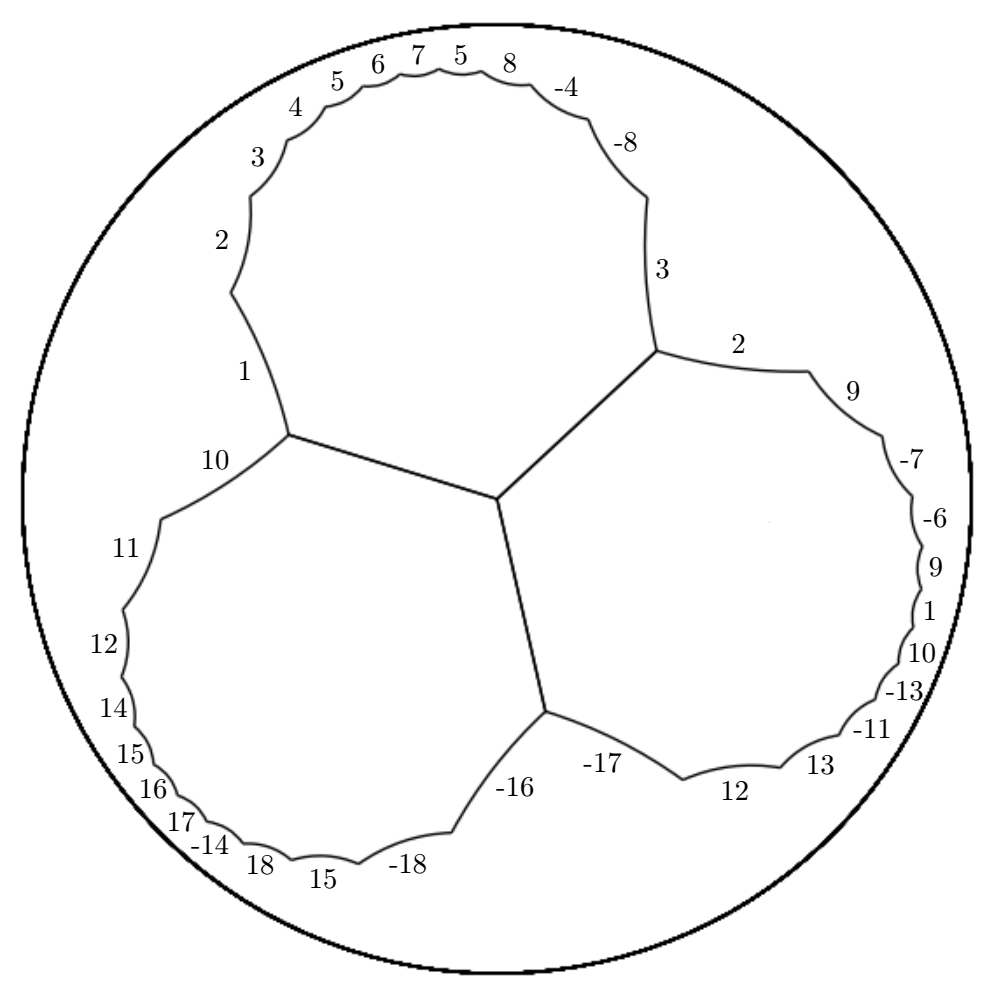}
\end{tabular}
\caption{A surface NEC group  of genus $g=4$ inside $\Delta^{\pm}(3,3,9)$ (on the left) and of surface NEC group of genus $g=6$ inside $\Delta^{\pm}(3,3,7)$ (on the right).}\label{ejemplos33n}
\end{center}
\end{figure}

We can cover the second case in a similar way. Compute first
$$
 (12g-24)\cdot 7 = 2k_2\cdot 7=[\Delta^{\pm}(2,3,7):K]=[\Delta^{\pm}(2,3,7): \Delta^{\pm}(3,3,7)][\Delta^{\pm}(3,3,7):K]=8[\Delta^{\pm}(3,3,7):K],
$$
hence  a genus $g$  compact non-orientable surface simultaneously $(3g-6)/4$- and $(6g-12)$- extremal would be uniformized by a torsion free proper NEC group of index $(3g-6)\cdot 7/2$ inside $\Delta^{\pm}(3,3,7)$. Now the lowest value of $g$ involved is $g=6$, which gives index 42, again the Hurwitz index of $\Delta^{\pm}(3,3,7)$. For a general value $g=2+4r$ ($r \ge 1$), the index is $42r$, a multiple of the Hurwitz index.
The NEC group $K_{42}$ generated by the side-pairings in Figure \ref{ejemplos33n} (right) uniformizes a compact non-orientable surface (with empty boundary) of genus $g=6$ that is $(3g-6)/4=3$- and $(6g-12)=24$-extremal. As before, higher index subgroups of $\Delta^{\pm}(3,3,7)$ can be constructed with Lemma  \ref{le:sub}.

\

Summarizing, we have proved the following result:

\begin{theorem}  If a compact non-orientable hyperbolic surface of genus $g$ is $k_j$-extremal for $m>1$ values $k_1, \ldots, k_m$ then $m=2$ and 
 one of the following holds:
\begin{itemize}
\item[1)] $\{k_1,k_2\}=\left\{2g-4,\displaystyle\frac{g-2}{2}\right\}$ 
\item[2)] $ \{k_1,k_2\}=\left\{6g-12,\displaystyle\frac{3g-6}{4}\right\}$ 
\end{itemize}
Conversely, for every even genus $g\ge 4$  there exists some compact non-orientable surface of genus $g$ which is $k_1$- and $k_2$- extremal for the values $k_1,k_2$ in case 1). Also, for every genus $g \equiv 2 \pmod 4$  there exists some compact non-orientable surface of genus $g$ which is $k_1$- and $k_2$- extremal for the values $k_1,k_2$ in case 2).
\end{theorem}

\section{Uniqueness of $k$-packings inside compact non-orientable $k$-extremal surfaces of genus $g$}

By the characterization of compact non-orientable $k$-extremal surfaces in terms of triangle groups given in Theorem \ref{th:char-triang}, we know that if such a surface $X\cong \mathbb{H}\slash K$ admits two extremal $k$-packings, then we must have two inclusions $$K\le \Delta_1^{\pm}=\Delta_1^{\pm}\left(2,3,\frac{6g+6k-12}{k}\right),\quad  K\le \Delta_2^{\pm}=\Delta_2^{\pm}\left(2,3,\frac{6g+6k-12}{k}\right)$$ both with index $12g+12k-24$. 

\

We know that there exists an isometry $\gamma\in\mathrm{PGL}(2,\mathbb{R})$ such that $\gamma\Delta_1^{\pm}\gamma^{-1}=\Delta_2^{\pm}\neq \Delta_1^{\pm}$. Then we have  $\gamma\in\mathrm{Comm}^{\pm}(\Delta_1^{\pm})$, the so called commensurator of $\Delta_1^{\pm}$, defined as $$\mathrm{Comm}^{\pm}(\Delta_1^{\pm}):=\{\alpha\in\mathrm{PGL}(2,\mathbb{R}):\ \Delta_1^{\pm}\cap \alpha\Delta_1^{\pm}\alpha^{-1}\ \text{has finite index inside}\ \Delta_1^{\pm}\ \text{and}\ \alpha\Delta_1^{\pm}\alpha^{-1}\}.$$ 

We recall the following facts:

\renewcommand{\labelenumi}{\Roman{enumi}}

\begin{enumerate}
\item If $\mathrm{Comm}^{\pm}(\Delta_1^{\pm})=\Delta_1^{\pm}$, and $\gamma\Delta_1^{\pm}\gamma^{-1}=\Delta_2^{\pm}$ then $\gamma\in\Delta_1^{\pm}$, and in particular $\Delta_1^{\pm}=\Delta_2^{\pm}$.

\item If $\mathrm{Comm}^{\pm}(\Delta^{\pm}(2,3,N))$ is a discrete group, then $\mathrm{Comm}^{\pm}(\Delta^{\pm}(2,3,N))=\Delta^{\pm}(2,3,N)$. The reason is that $\Delta^{\pm}(2,3,N)$ is maximal for every $N$ (see \cite{Est_Izq_2006}), 

\item If the canonical Fuchsian subgroup $\Delta^{+}(2,3,N)$ of the extended triangle group $\Delta^{\pm}(2,3,N)$ is non-arithmetic, then $\mathrm{Comm}^{\pm}(\Delta^{\pm}(2,3,N))$ is a discrete group (because of Theorem 3.1 in \cite{Girondo_Nakamura_2007}, based on \cite{Margulis_1991}).

\item (\cite{Takeuchi_1977_a}) A Fuchsian triangle group $\Delta^{+}(2,3,N)$ is non-arithmetic if and only if $$N\not\in\{7,8,9,10,11,12,14,16,18,24,30\}.$$

\end{enumerate}

We can combine the previous results to see that for compact non-orientable  $k$-extremal surfaces whose NEC group $K$ is contained in a non-arithmetic triangle group $\Delta^{\pm}$, the $k$-packing will be unique. This happens as long as $$\frac{6g+6k-12}{k}\not\in\{7,8,9,10,11,12,14,16,18,24,30\}$$

Therefore we have:

\begin{theorem} \label{t4} If $X$ is a compact non-orientable  $k$-extremal surface of genus $g$ and 
$$
N=\frac{6g+6k-12}{k} \not\in\{7,8,9,10,11,12,14,16,18,24,30\}
$$ then the extremal $k$-packing is unique in $X$.
\end{theorem}

For example, for $k=1$  Theorem \ref{t4}  says that $1$-packings in compact non-orientable  surfaces will be unique if $6g-6$ is not in the above list of numbers, which precisely occurs for $g>6$ (consistent with \cite{Girondo_Nakamura_2007}). In the series of papers \cite{Girondo_Nakamura_2007}, \cite{Nakamura_2009}, \cite{Nakamura_2012}, \cite{Nakamura_2013} and \cite{Nakamura_2016} all $1$-extremal surfaces of genus $g$ from 3 to 6 were studied in detail, and in particular the existence of extremal surfaces with more than one extremal disc was explicitly shown for these values of $g$.

\

Now, suppose that there exists a compact non-orientable primitive $k_N$-extremal surface $X$ 
of genus $g_N$ with multiple extremal $k_N$-packings. Denote $N=\displaystyle\frac{6g_N+6k_N-12}{k}$. As mentioned above, the NEC group $K$ uniformizing $X$ is contained simultaneously in 
$\Delta_1^{\pm}\left(2,3,N\right)$ and in some conjugate  $\gamma\Delta_1^{\pm}\left(2,3,N\right) \gamma^{-1}$, and the same happens obviously for any subgroup of $K$. By Lemma \ref{le:sub} we would deduce the existence of compact non-orientable imprimitive $k$-extremal surfaces  
of genus $g$ with multiple extremal $k$-packings for every pair $(k, g) \in L_N$ (recall the definition of $L_N$ from Section \ref{sec:discpack}).

\

It follows then that constructing an example of a primitive $k_N$-extremal surface of genus $g_N$ with multiple extremal $k_N$-extremal packings 
for the eleven possible values of $N$ in Theorem \ref{t4}, we would have proven the following converse result:

\begin{theorem}
If $k\ge 1$ and $g\ge 3$ are such that 
$$
\frac{6g+6k-12}{k} \in\{7,8,9,10,11,12,14,16,18,24,30\}
$$
then there exists a $k$-extremal compact non-orientable surface with multiple extremal $k$-packings. 
\end{theorem}

Four of the eleven surfaces we need have occurred already in the literature, namely the ones corresponding to the cases for which $k_N=1$ (see \cite{Girondo_Nakamura_2007}, \cite{Nakamura_2009}, \cite{Nakamura_2012}, \cite{Nakamura_2013} and \cite{Nakamura_2016}). These cases correspond to $N=12, 18, 24$ or $30$. 

\

We  simply present here an explicit example of the remaining cases. The reason explaining why all them have of a second $k$-packing is the existence of an automorphism of the surface that does not fix the first $k$-packing we start with. These automorphisms are induced by order 2 elliptic transformations not belonging to the corresponding triangle group, and this is why they are very difficult to detect. We refer to \cite{arxiv} for the reader interested in the (lengthy, and highly non-trivial) computational details of how can one construct such examples.

\

Figure \ref{N7ejemplo_varios} shows two isometric fundamental regions for the NEC group uniformizing the surface $X_7$ we already showed in Figure \ref{ejemplos}, corresponding to two different extremal 6-packings. The point $\mathrm{P}$ is fixed by an order two orientation preserving automorphism that permutes both 6-packings. Surprisingly enough, the two sets of side-pairing transformations on both sides of the figure (strong-lined fundamental region on the left, light-lined fundamental region on the right) generate \emph{exactly} the same group.

\begin{figure}[!h]
\begin{center}
\includegraphics[width=\textwidth]{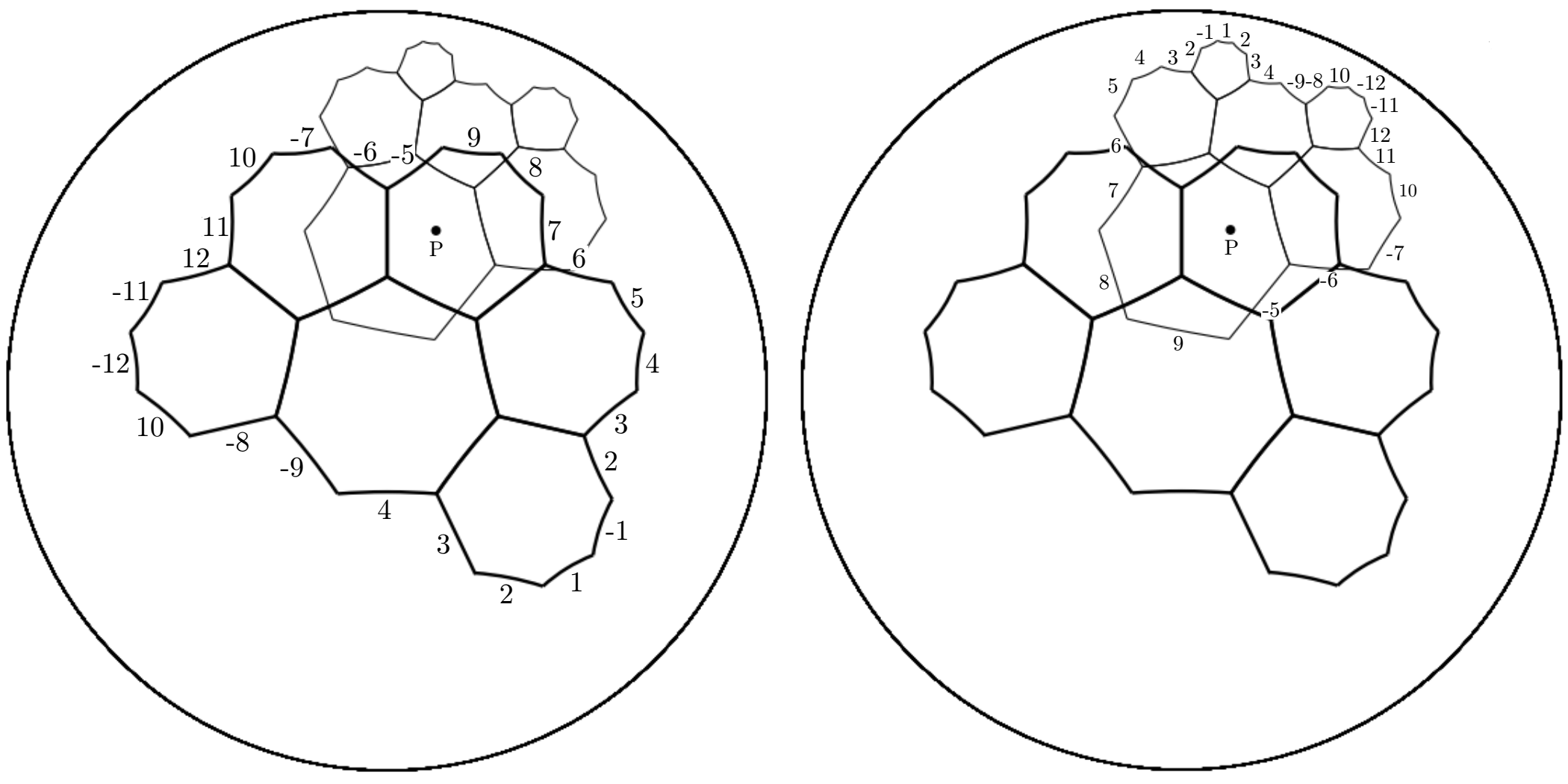} 
\caption{A surface of genus $g=3$ with two extremal 6-packings. An orientation preserving involution 
fixing $\mathrm{P} \simeq 0.129 + 0.422 i$ transposes both 6-packings.}\label{N7ejemplo_varios}
\end{center}
\end{figure}

Figure \ref{ejemplos_varios} shows the required examples for the cases $N=8, 9, 10, 11, 14$ and $16$. In each case the point $\mathrm{P}$ is fixed by the orientation preserving involution that produces a second extremal packing. The interested reader can check that the order 2 elliptic transformation fixing the point $\mathrm{P}$, whose approximate numerical expression is given in Table \ref{table_p}, normalizes the group generated by the side pairing transformations in all the cases.
We do not depict the fundamental regions associated  to the second
 \emph{hidden} extremal packings (analogous to the light-lined union of polygons in Figure \ref{N7ejemplo_varios}) in order to make the figure simpler.
 
 \begin{table}[!htbp] 
 \begin{center}
 \begin{tabular}{|c|c|c|c|c|c|c|}
 \hline
 $(k,g)$ & (3,3) & (2,3) & (3,4) & (6,7) & (3,6) & (3,7) \\  \hline
 $\mathrm{P}\simeq$ & 0.433+0.113i & 0.677+0.415i & 0.310-0.225i &  -0.277-0.191i & -0.360 + 0.201i &  0.284 - 0.131i \\ \hline
 \end{tabular}
 \end{center}
 \caption{A numerical approximation of the point $\mathrm{P}$ in the surfaces depicted in Figure \ref{ejemplos_varios}.} \label{table_p}
 \end{table}

\begin{figure}[!hptb]
\begin{center}
\begin{tabular}{cc}
\includegraphics[width=0.4\textwidth]{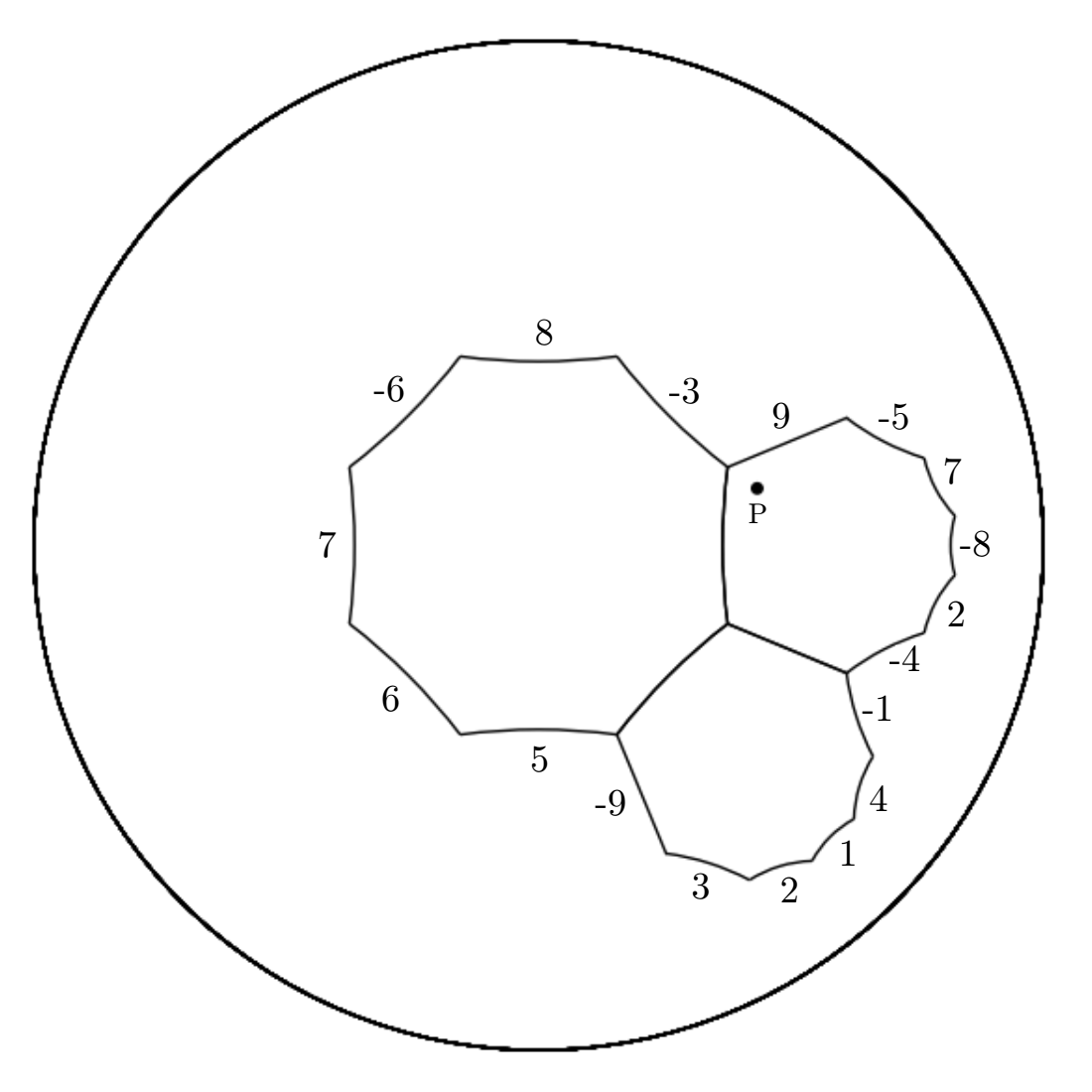} & \includegraphics[width=0.4\textwidth]{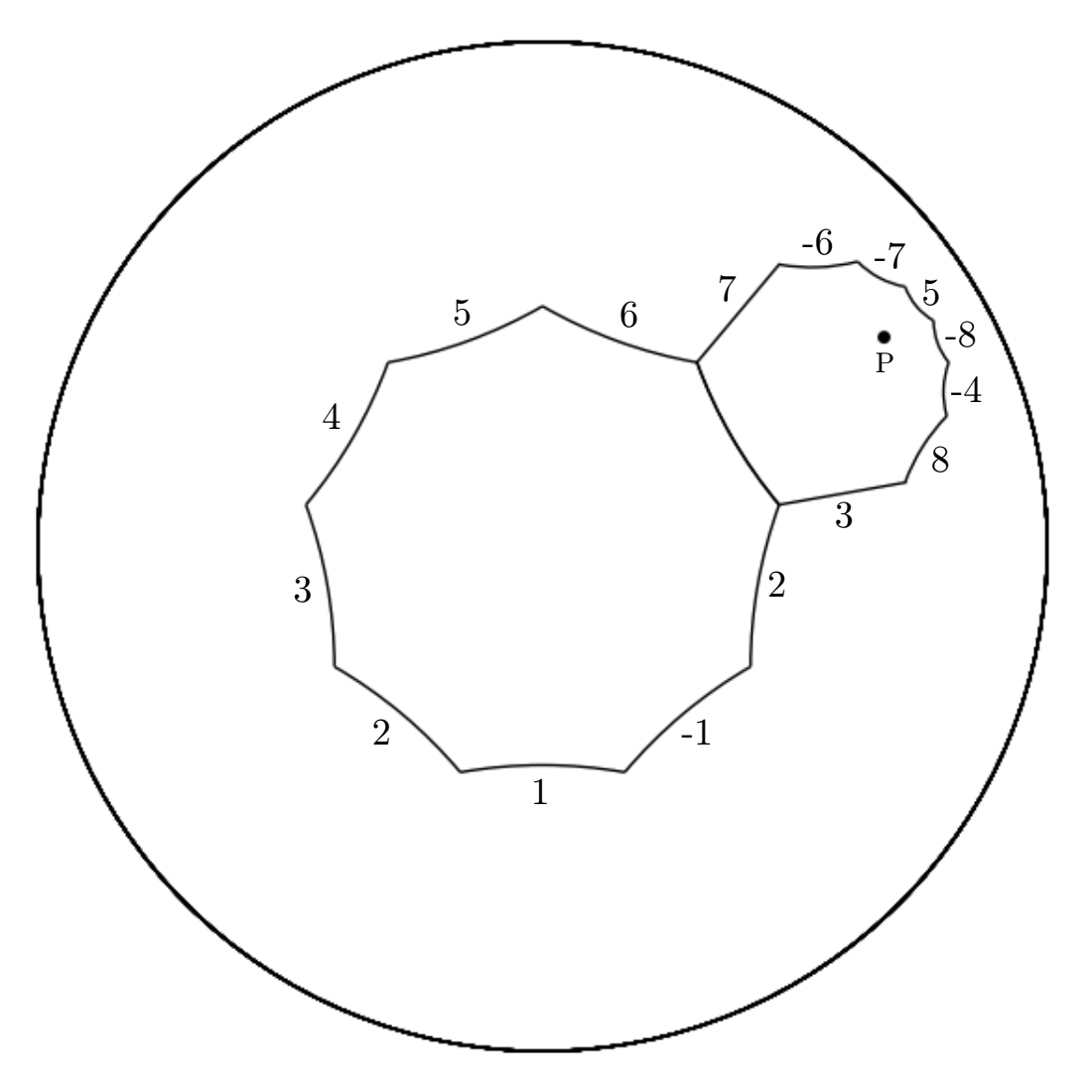}\\
\includegraphics[width=0.4\textwidth]{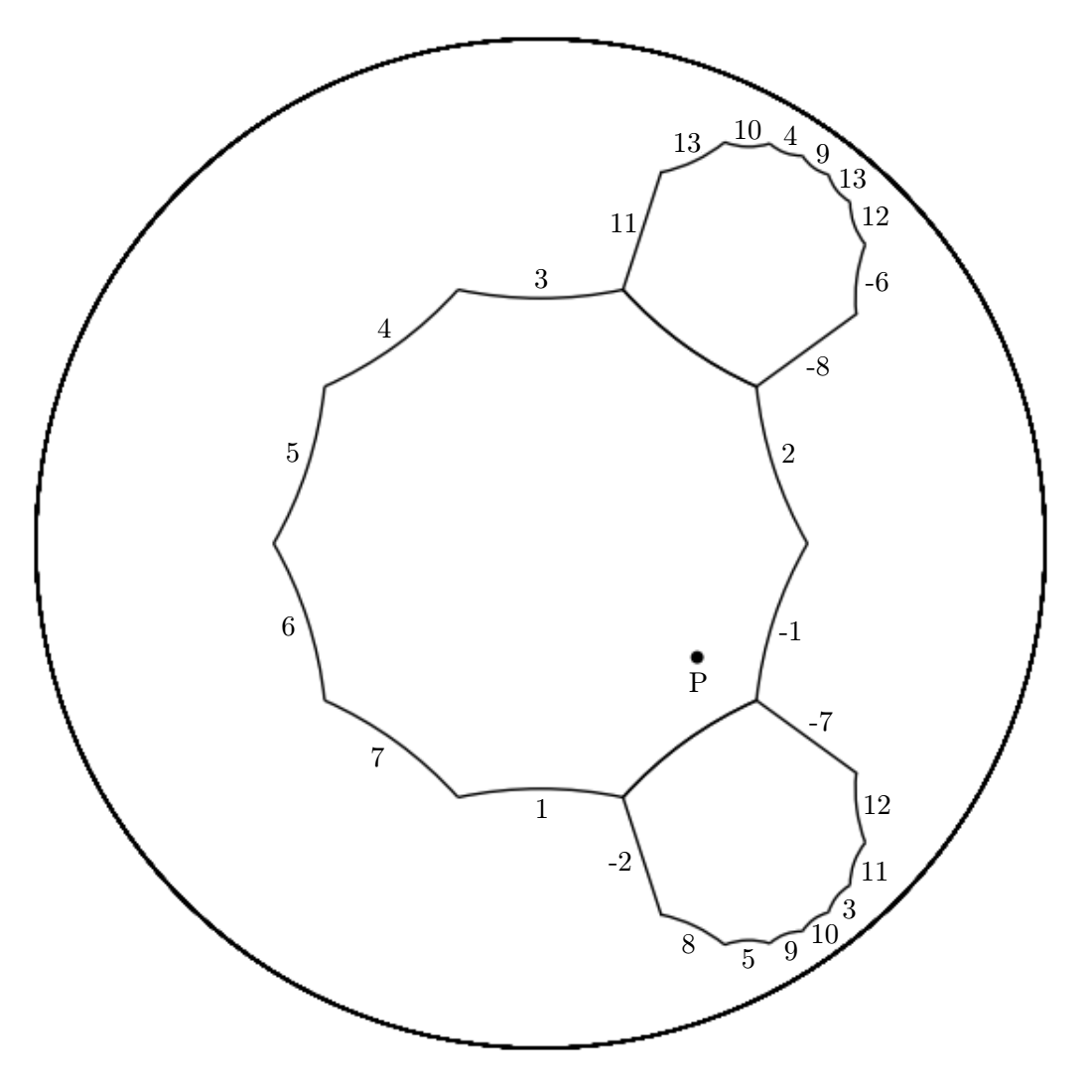} & \includegraphics[width=0.4\textwidth]{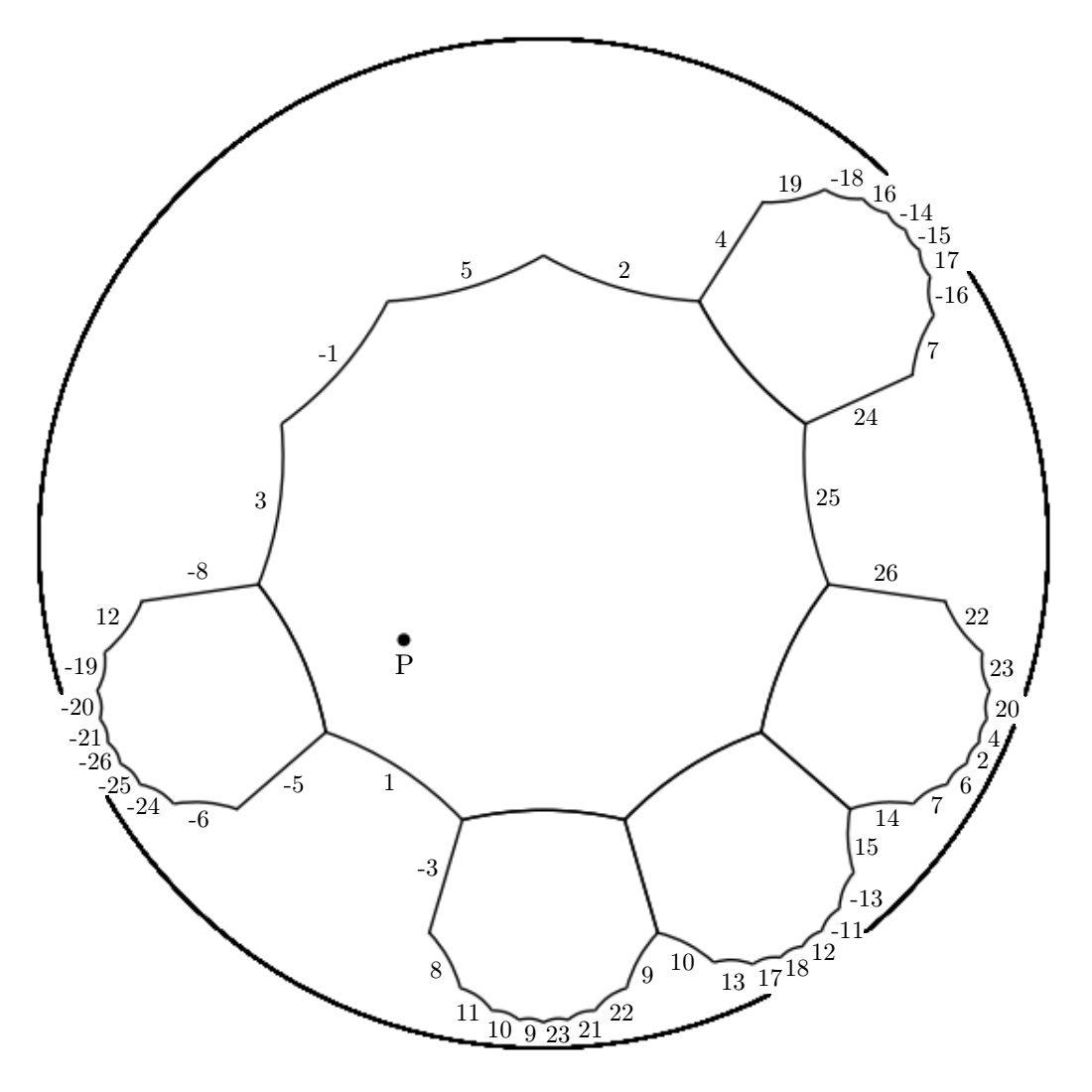}\\
\includegraphics[width=0.4\textwidth]{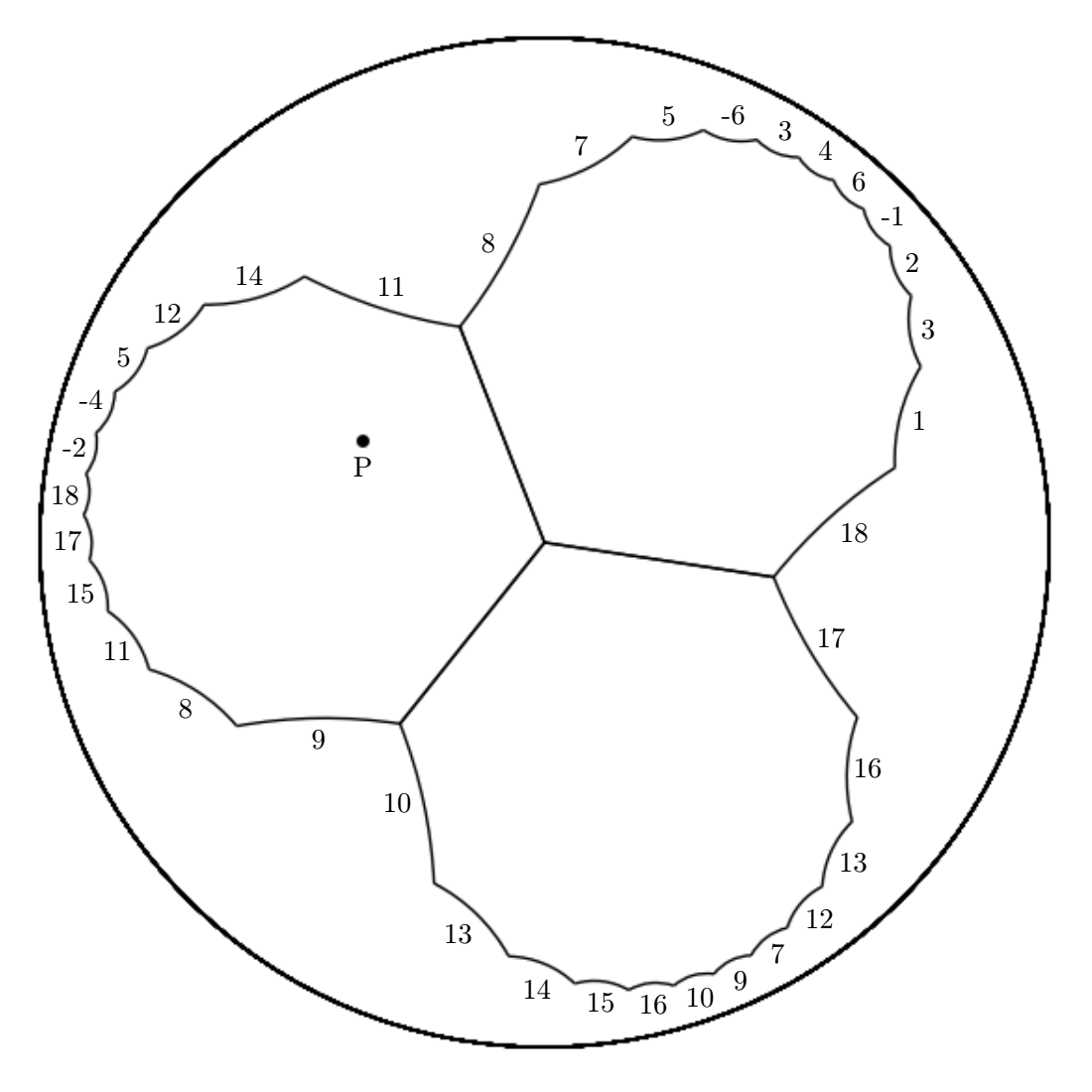} & \includegraphics[width=0.4\textwidth]{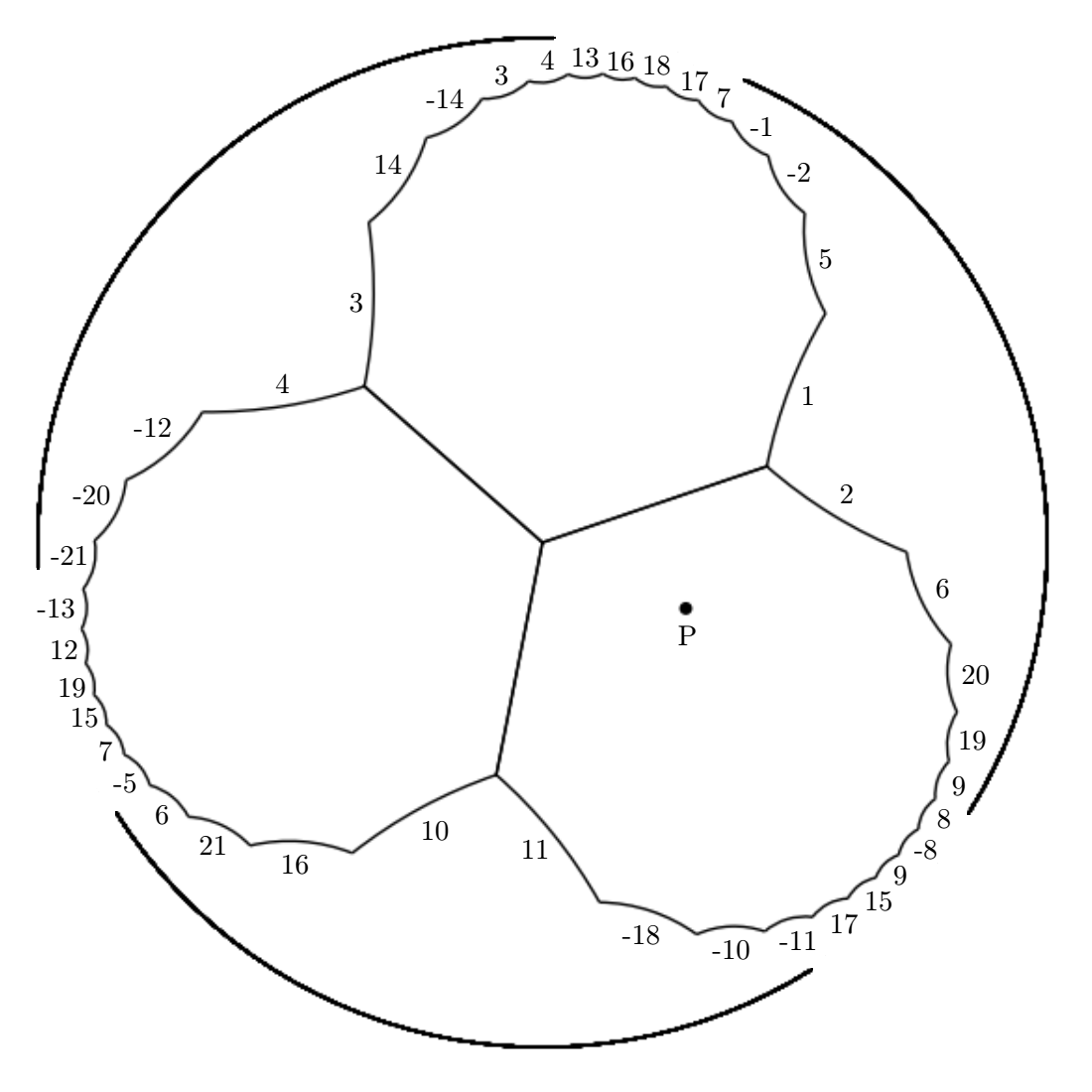}
\end{tabular}
\caption{A colection of compact hyperbolic $k$-extremal surfaces of genus $g$ for the pairs (from left to right, top to bottom) $(k,g)=(3,3), (2,3), (3,4), (6,7), (3,6), (3,7)$.}\label{ejemplos_varios}
\end{center}
\end{figure}

\

It is not clear if two extremal packings within a given surface must be always related by an automorphism. This is certainly the case for $k=1$,  but determining whether or not this is a general result remains an interesting open problem.

\section{Non-orientable $k$-extremal surfaces and their double cover}

 The existence of a compact non-orientable $k$-extremal surface of genus $g$ is equivalent to the existence of a proper surface NEC group $K$ with index $12g+12k-24$ inside an extended triangle group $\Delta^{\pm}\left(2,3,\frac{6g+6k-12}{k}\right)$. For such $K$ we can consider its canonical Fuchsian subgroup $K^{+}$, which is contained in $K$ with index $2$, and then we have the following chain of inclusions
 
 \
 
\centerline{ \xymatrix{
 & \Delta^{\pm} & \\
 K\ar[ur]^{6g+6k-12} & & \Delta^{+}=\Delta^{\pm}\cap \mathrm{PSL}(2,\mathbb{R}) \ar[ul]_{2}  \\
 & K^{+}\ar[ur]_{\ \ 6g+6k-12}\ar[ul]^{2} & 
}}
 
\

It follows that $K^{+}$ uniformizes a compact Riemann surface of genus $g'$ given by the formula
$$2g'-2=[2\cdot(6g+6k-12)]\left(2\cdot 0-2+\left(1-\frac{1}{2}\right)+\left(1-\frac{1}{3}\right)+\left(1-\frac{1}{\frac{6g+6k-12}{k}}\right)\right)$$ 
which yields $g'=g-1$ after a straightforward computation.\\

In other words, given a compact non-orientable $k$-extremal surface of genus $g$ there is an associated compact Riemann surface with genus $g-1$, uniformized by a Fuchsian group $K^{+}$ included in $\Delta^{+}\left(2,3,\frac{6g+6k-12}{k}\right)$. This Riemann surface is certainly a $k'$-extremal surface,  and since 
$$\frac{6g+6k-12}{k}=\frac{12(g-1)+6k'-12}{k'}$$ 
we conclude that $k'=2k$. We have proved the following

\begin{theorem} If $X=\mathbb{H}\slash K$ is a compact non-orientable $k$-extremal surface of genus $g$, then $K^{+}=K\cap \mathrm{PSL}(2,\mathbb{R})$ uniformizes a compact $2k$-extremal Riemann surface $\mathbb{H}\slash K^{+}$ of genus $g-1$ which is a double cover of $X$.
\end{theorem}

\bibliography{references}  
\bibliographystyle{alpha}

\end{document}